\documentclass[12pt]{article}

\usepackage[margin=1.12in]{geometry}

\usepackage{amssymb,amsmath}
\usepackage{amsthm,enumitem}
\usepackage{hyperref}
\usepackage{mathrsfs}
\usepackage{textcomp}
\hypersetup{
    colorlinks,%
    citecolor=black,%
    filecolor=black,%
    linkcolor=black,%
    urlcolor=black
}

\usepackage{verbatim}

\usepackage{sectsty}

\makeatletter

\newdimen\bibspace
\renewenvironment{thebibliography}[1]{%
 \section*{\refname 
       \@mkboth{\MakeUppercase\refname}{\MakeUppercase\refname}}%
     \list{\@biblabel{\@arabic\c@enumiv}}%
          {\settowidth\labelwidth{\@biblabel{#1}}%
           \leftmargin\labelwidth
           \advance\leftmargin\labelsep
           \itemsep\bibspace
           \parsep\z@skip     %
           \@openbib@code
           \usecounter{enumiv}%
           \let\p@enumiv\@empty
           \renewcommand\theenumiv{\@arabic\c@enumiv}}%
     \sloppy\clubpenalty4000\widowpenalty4000%
     \sfcode`\.\@m}
    {\def\@noitemerr
      {\@latex@warning{Empty `thebibliography' environment}}%
     \endlist}

\makeatother

\makeatletter

\newtheorem{thm}{Theorem}[section]
\newtheorem{lem}[thm]{Lemma}

\newtheorem{cor}[thm]{Corollary}
\newtheorem{rem}[thm]{Remark}

\theoremstyle{definition}

\newtheorem{example}[thm]{Example}

\def\XXint#1#2#3{{\setbox0=\hbox{$#1{#2#3}{\int}$}
  \vcenter{\hbox{$#2#3$}}\kern-.5\wd0}}

\newcommand{\al}{\alpha}                \newcommand{\lda}{\lambda}
\newcommand{\om}{\Omega}                \newcommand{\pa}{\partial}
\newcommand{\va}{\varepsilon}           \newcommand{\ud}{\mathrm{d}}
\newcommand{\be}{\begin{equation}}      \newcommand{\ee}{\end{equation}}
                 
\newcommand{\Lda}{\Lambda}              
\newcommand{\R}{\mathbb{R}}

\begin{document}

\title{\textbf{Optimal regularity and fine asymptotics for the porous medium equation in bounded domains}\bigskip}

\author{\medskip  Tianling Jin,\quad Xavier Ros-Oton, \quad
Jingang Xiong}

\date{\today}

\maketitle

\begin{abstract} 
We prove the optimal global regularity of nonnegative solutions to the porous medium equation in smooth bounded domains with the zero Dirichlet boundary condition after certain waiting time $T^*$. 
More precisely, we show that solutions are $C^{2,\alpha}(\overline\Omega)$ in space, with $\alpha=\frac1m$, and $C^\infty$ in time (uniformly in $x\in \overline\Omega$), for $t>T^*$.

Furthermore, this allows us to refine the asymptotics of solutions for large times, improving the best known results so far in two ways: we establish a faster rate of convergence $O(t^{-1-\gamma})$, and we prove that the convergence holds in the $C^{1,\alpha}(\overline\Omega)$ topology.

\medskip


\medskip


\end{abstract}

\section{Introduction}

Let $\om$ be a bounded smooth domain in $\R^n$ with $n\ge 1$ and let $m>1$. Consider the porous medium equation (PME):
\be \label{eq:main}
\left\{\begin{array}{rcll}
\pa_t u -\Delta (u^m) & = & 0  \qquad \mbox{ in }\ \om \times (0,\infty) \\
u & = & 0  \qquad \mbox{ on }\ \pa \om \times (0,\infty)\\
u(x,0) & = & u_0 (x)\ge 0. 
\end{array}\right.
\ee

The equation \eqref{eq:main} is a slow diffusion equation, which means that if $u_0$ is compactly supported in $\Omega$, then the solution $u(\cdot,t)$ with such initial data will still be compactly supported in $\Omega$ at least for a short time, and thus, $\partial\{u>0\}$ is a free boundary. 
Due to the degeneracy of the equation near $\{u=0\}$, the initial boundary value problem \eqref{eq:main} does not possess in general a classical solution (see, e.g., Oleinik-Kalashnikov-Chzou \cite{OKC}). Thus, it is necessary to work with a suitable class of weak solutions. We say that $u$ is a nonnegative weak solution of \eqref{eq:main} if $u\in C([0,+\infty): L^1(\Omega))$ such that $u^m\in L^2_{loc}([0,+\infty):H^1_0(\Omega))$ and $u$ satisfies \eqref{eq:main} in the sense of distribution. The initial boundary value problem \eqref{eq:main} is then well-posed for such weak solutions.
H\"older continuity of the weak solutions  and their free boundaries were proved by Caffarelli-Friedman \cite{CF}. Their higher regularities were proved, for example, by Caffarelli-V\'azquez-Wolanski \cite{CVW}, Caffarelli-Wolanski \cite{CW}, Daskalopoulos-Hamilton \cite{DH}, Koch \cite{Koch}, Daskalopoulos-Hamilton-Lee \cite{DHL} and Kim-Lee \cite{KimLee} under a non-degeneracy condition of the initial data. Kienzler-Koch-V\'azquez \cite{KKV} proved the smoothness of the weak solution and the free boundary for all large times. 
The one spatial dimension case has been studied earlier in, e.g., Aronson \cite{A, A70}, Knerr \cite{Knerr77},  Aronson-Caffarelli-V\'azquez \cite{ACV}, Aronson-V\'azquez \cite{AV87}, H\"ollig-Kreiss \cite{HK600}, and Angenent \cite{Angenent}, from which we know that the free boundary interface is eventually analytic for the equation \eqref{eq:main} posed in $\R\times(0,+\infty)$ with compactly supported initial data. 
Several universal estimates for the porous medium equation have also been obtained by Aronson-B\'enilan \cite{AB} and Dahlberg-Kenig \cite{DKenig}. 
We refer to Daskalopoulos-Kenig \cite{DaK} and V\'azquez \cite{Vaz} for more references on the initial boundary value problem for the PME.

An explicit solution to the PME which often plays an important role is the so-called \emph{friendly giant solution}, given by
\be\label{friendly-giant}
U(x,t) := t^{-\frac{1}{m-1}} S(x),
\ee
where $S$ is the unique nontrivial nonnegative solution of
\begin{equation}\label{eq:stationary}
-\Delta (S^m)={\textstyle\frac{1}{m-1}}S \quad\mbox{in }\Omega\quad \mbox{and}\quad S=0\quad\mbox{on }\partial\Omega.
\end{equation}
The existence and uniqueness of this $S$ was proved by Aronson-Peletier \cite{AP1981}. It follows from Dahlberg-Kenig \cite{DKenig} that for every weak solution $u$ of \eqref{eq:main}, it satisfies
\begin{equation}\label{eq:universalupperbound}
u\le U\quad\mbox{in }\overline\Omega\times[0,+\infty).
\end{equation}
See also Proposition 1.3 of V\'azquez \cite{Vaz2004}. Notice that $U(x,0)=+\infty$ in $\Omega$, while $U\asymp d^{1/m}$ for all $t>0$, where $d(x)={\rm dist}(x,\partial\Omega)$.

\subsection{Optimal regularity of solutions}

Aronson-Peletier \cite{AP1981} proved that there exists some waiting time $T^*$ such that any nontrivial solution $u$ of \eqref{eq:main} will be positive in $\Omega$ for all time afterwards, and moreover
\begin{equation}\label{T*}
u\asymp d^{1/m}\asymp S \quad \textrm{for}\quad t>T^*.
\end{equation}
{A uniform estimate of $T^*$ is given by Bonforte-V\'azquez \cite{BV15}.  See also Bonforte-Figalli-V\'azquez \cite{BFV18}.}
This, combined with interior regularity estimates, implies that $u(\cdot,t)\in C^{1/m}(\overline\Omega)$ for all $t>T^*$  (see \cite{BFR17}) and that 
\be\label{eq-Lip-x}
u^m(\cdot,t)\in {\rm Lip}(\overline\Omega) \quad \textrm{for all}\  t>T^*.
\ee
A long standing open question in this context is the following: 
\[\textit{Are solutions to the PME  \eqref{eq:main} classical up to the boundary for $t>T^*$?}\]
\[\textit{In other words, is $u^m(\cdot,t)\in C^2(\overline\Omega)$ for all $t>T^*$?}\]
The goal of this paper is to answer this question, and to obtain the optimal regularity near $\partial\Omega$ for solutions to \eqref{eq:main}.
Furthermore, as explained below, this will allow us to obtain finer asymptotics for large times $t\to\infty$.

Our first result reads as follows.

\begin{thm}\label{eq:mainpme1}
Let $\om\subset \R^n$ be any bounded smooth domain, $m>1$, and $u_0\in L^1(\Omega)$ with $u_0\not\equiv0$. 
Let $u$ be the nonnegative weak solution of \eqref{eq:main}, and let $T^*$ be the first time for which \eqref{T*} holds.
Then,
\be \label{eq-optimal-x}
u^m(\cdot,t) \in C^{2+\frac{1}{m}}(\overline\Omega)\quad\mbox{for all }t>T^*,
\ee
and
\[u^m(x,\cdot)\in C^\infty((T^*,\infty)) \quad \textrm{uniformly in} \ x\in\overline\Omega.\]
Moreover, \eqref{eq-optimal-x} does not hold in general with exponent $2+\frac1m+\varepsilon$ for any $\varepsilon>0$.

Furthermore, we also have 
\[
\partial_t^k (u^m)(\cdot,t) \in C^{2+\frac{1}{m}}(\overline\Omega)\quad\mbox{for all }t>T^*
\]
for all $k\in \mathbb N$.
\end{thm}

Notice that, for any $\varepsilon>0$, the friendly giant solution \eqref{friendly-giant} does \emph{not} belong to $C^{2+\frac1m+\varepsilon}(\overline\Omega)$ for any $t>0$.
Thus, this already shows the optimality of our result above.

One can also deduce from \eqref{eq-optimal-x} another optimal regularity result for $u$, as follows.

\begin{cor}\label{cor-u}
Let $\om\subset \R^n$ be any bounded smooth domain, $m>1$, and $u_0\in L^1(\Omega)$ with $u_0\not\equiv0$. 
Let $u$ be the nonnegative weak solution of \eqref{eq:main}, and let $T^*$ be the first time for which \eqref{T*} holds.
Then,
\be \label{eq-optimal-u}
\frac{u}{S} \in C^{1+\frac{1}{m}}(\overline\Omega)\quad\mbox{for all }t>T^*,
\ee
where $S$ is given by \eqref{eq:stationary}.

Moreover, \eqref{eq-optimal-u} does not hold in general with exponent $1+\frac1m+\varepsilon$ \ for any $\varepsilon>0$.
\end{cor}

It is interesting to notice that, while in the free boundary case ---i.e., when $\Omega=\R^n$ in \eqref{eq:main} and the set $\{u>0\}$ is moving in time--- it was shown in Kienzler-Koch-V\'azquez \cite{KKV} that $u^{m-1}$ and $u^{m-1}/d$ are $C^\infty$ up to the free boundary for large times, in case of Dirichlet conditions in bounded domains it is \emph{not} true that $u^m$ or $u^m/d$ are $C^\infty$ up to the boundary for large times\footnote{Here, $d$ denotes the distance to the boundary, modified inside the domain so that it is $C^\infty(\overline\Omega)$.}.
This is actually false even for the friendly giant solution\footnote{Since $S$ solves \eqref{eq:stationary}, and $S\asymp d^{1/m}$, then the Laplacian of $S^m$ is not $C^{\frac1m+\varepsilon}$, and thus $S^m$ cannot be $C^{2+\frac1m+\varepsilon}$ for any $\varepsilon>0$.}, which satisfies that $u^m$  is $C^{2+\frac1m}$, and not $C^{2+\frac1m+\varepsilon}$  for any $\varepsilon>0$. 
In particular, for the same reasons, we have that $u^m/d$  is always $C^{1+\frac1m}$, however it is not $C^{1+\frac1m+\varepsilon}$ for any $\varepsilon>0$.
One might then wonder if  $u^m/S^m$  could be more regular (say, $C^\infty$), but it turns out to be false, too, as stated in Corollary \ref{cor-u} above.

Thus, these two results completely answer the question of optimal boundary regularity of solutions to the Dirichlet problem for the PME.

\subsection{Long time behavior}

As said before, Aronson-Peletier \cite{AP1981} proved that there exists a waiting time $T^*\ge 0$ for which \eqref{T*} holds.
Concerning the large time behavior $t\to\infty$, they showed that there exists $\tau>0$ such that
\begin{equation}\label{eq:bounds}
u(x,t)\ge (\tau+t)^{-\frac{1}{m-1}}S(x)\quad\mbox{in }\overline\Omega\times(T^*,+\infty),
\end{equation}
where $S$ is given by \eqref{eq:stationary}. Consequently, they proved the stability of the friendly giant solution \eqref{friendly-giant} in the sense that 
\begin{equation}\label{eq:stability}
\left\|\frac{u(\cdot,t)}{U(\cdot,t)}-1\right\|_{L^\infty(\Omega)}\le \frac{C}{t}\quad\mbox{for  }t>T^*,
\end{equation}
where $C>0$ depends only on $n,m,u_0$ and $\Omega$. 
The decay rate in \eqref{eq:stability} is optimal by considering the particular solution $u_s(x,t)=(s+t)^\frac{1}{1-m}S(x)$ for arbitrary $s>0$. 
See {Theorem 1.1} in V\'azquez \cite{Vaz2004} for another similar stability result of the separable solution to the porous medium equation \eqref{eq:main} {with general initial data and general domains,  Theorem 5.8 in Bonforte-Grillo-V\'azquez \cite{BGV} for another proof using a new entropy method, and also Theorems 2.5 and 2.6 in Bonforte-Sire-V\'azquez \cite{BSV15} for more general porous medium type equations.}

Here, thanks to our fine boundary estimates from Theorem \ref{eq:mainpme1}, we can establish finer estimates for the long time behavior of solutions to the PME.

 \begin{thm}\label{eq:mainpme}
Let $\om\subset \R^n$ be any bounded smooth domain, $m>1$, and $u_0\in L^1(\Omega)$ with $u_0\not\equiv0$. 
Let $u$ be the nonnegative weak solution of \eqref{eq:main}, and $T^*$ be the first time for which \eqref{T*} holds. Let $\delta>0$.

Then, there exist constants $A_1\ge 0$, $\tau^*\geq 0$, $C_1>0$ and $\gamma>0$ such that
\begin{equation}\label{eq:decayestimateu}
\left\|t^\frac{m}{m-1}u^m(\cdot,t)-S^m+\frac{A_1 }{t} S^m\right\|_{C^{2+\frac{1}{m}}(\overline\Omega)}\le \frac{C_1}{t^{1+\gamma}} \quad\mbox{for all }t> T^*+\delta
\end{equation}
and
\begin{equation}\label{eq:decayestimateu2}
\left\|\frac{u(\cdot,t)}{U(\cdot,\tau^*+t)}-1 \right\|_{C^{1+\frac{1}{m}}(\overline\Omega)} 
\le \frac{C_1}{t^{1+\gamma}} \quad\mbox{for all }t> T^*+\delta,
\end{equation}
where $S$ and $U$ are given by \eqref{eq:stationary} and \eqref{friendly-giant}, respectively,  $C_1$ depends only on $n,m,\delta,\Omega$ and $u_0$, $A_1$ and $\tau^*$ depend only on $n,m,\Omega$ and $u_0$, while $\gamma$ depends only on $n$, $m$, and $\Omega$.
In particular, 
\[\textrm{in the dimension}\quad n=1 \quad\textrm{we have}\quad \gamma=1.\]

Furthermore, for any $\ell\in\mathbb{N}$, there exists $C_2>0$ depending only on $n,m,\delta,\Omega, \ell$ and $u_0$ such that
\begin{equation}\label{eq:regularityestimateu}
\|\partial_t^\ell (u^m)(\cdot,t)\|_{C^{2+\frac{1}{m}}(\overline\Omega)}\le C_2 t^{-\frac{m}{m-1}-\ell}\quad\mbox{for all }t> T^*+\delta.
\end{equation}
\end{thm}

This result improves substantially the best known results  so far for the long time behavior of solutions to the PME.
Indeed, it not only improves Aronson-Peletier's stability result \eqref{eq:stability} from the $L^\infty(\Omega)$ topology to the $C^{1+\frac{1}{m}}(\overline\Omega)$ topology, but also gives a faster rate of convergence  than~\eqref{eq:stability}, which we expect to be optimal. 

Indeed, on the one hand, recall that the $C^{1+\frac{1}{m}}(\overline\Omega)$ regularity of the relative error $u/U$ is optimal; see Corollary~\ref{cor-u}.

On the other hand, as shown in the proof of Theorem \ref{eq:mainpme}, the constant $\gamma$ is determined by the second eigenvalue of the linearized operador of the equation $-\Delta \Theta=\frac{1}{m-1}\Theta^{1/m}$ in $\Omega$, with the zero Dirichlet condition on $\partial\Omega$.
In particular, $\gamma>0$ depends strongly on the domain $\Omega$, and in general we do not expect any lower bound on $\gamma>0$ like the one we prove for the case of dimension $n=1$.

Finally, it is interesting to notice that, as explained in Remark~\ref{rem:4.2}, by adding extra terms (involving higher eigenfunctions of such linearized operator) in the expansion \eqref{eq:decayestimateu}, one could expand to arbitrary orders.

\subsection{Strategy of the proof}

As in many nonlinear PDE problems, in order to establish higher regularity of solutions one would like to use Schauder-type estimates and a bootstrap argument but, for this, some initial regularity is needed.

In our context, the a priori Schauder-type estimates we need were established by Kim-Lee \cite{KimLee} ---for compatible initial data satisfying some nondegeneracy conditions---, using the methods of Daskalopoulos-Hamilton \cite{DH}.

The initial regularity we need in order to use the Schauder-type estimates, however, is  
\be \label{eq:bootstrap-stp}
\mbox{the H\"older continuity of } \frac{u^m}{d} \mbox{ on }\overline\Omega\times(T^*,+\infty).
\ee  
This does not follow from previously known results, which give at best that such quotient is bounded --- recall \eqref{eq-Lip-x}.

We prove \eqref{eq:bootstrap-stp} by using the ideas of a recent work of the first and last authors \cite{JX22} on the fast diffusion equation (corresponding to $m\in (0,1)$ in \eqref{eq:main}), which solved a problem raised by Berryman and Holland \cite{BH} in 1980. 
The rough idea is that we need to develop a De Giorgi iteration for a singular and degenerate nonlinear parabolic equation.
Moreover, while the equation in \cite{JX22} corresponds to the case $m\in (0,1)$, here we need to treat the case $m>1$.

After proving  \eqref{eq:bootstrap-stp}, we establish the all time regularity of solutions with compatible initial data --- by using a boostrap argument and the Schauder estimates from \cite{KimLee} ---, and finally we need an appropriate approximation argument to establish the eventual regularity for solutions with general initial data.

Finally, to prove Theorem \ref{eq:mainpme}, we find an equation for $t^{\frac{m}{m-1}}(u^m-U^m)$ (after a change of the time variable of the form $\tau:=\log t$), and prove that, up to errors that decay faster as $\tau\to\infty$, the solution is well approximated by an expansion involving the eigenfunctions of the linearized operator of the equation $-\Delta \Theta=\frac{1}{m-1}\Theta^{1/m}$ in $\Omega$, with the zero Dirichlet condition on~$\partial\Omega$.
Since the first eigenfunction turns out to be $\Theta$ itself (this gives the constants $A_1$ or $\tau^*$), then we get a rate of convergence which is dictated by the second eigenvalue of such operator, and this gives the constant $\gamma>0$.

\subsection{Related works}

In the setting of uniformly parabolic equations,  boundary estimates of type \eqref{eq:bootstrap-stp}  have been studied in, e.g., Krylov \cite{Kr}, Fabes-Garofalo-Salsa \cite{FGS} and Fabes-Safonov \cite{FabS}. 
For boundary estimates of solutions to certain degenerate or singular nonlinear parabolic equations related to \eqref{eq:main}, we refer to the recent papers  Kuusi-Mingione-Nystr\"om \cite{KMN}, Avelin-Gianazza-Salsa \cite{AGS} and references therein.

In case of the fast diffusion equation ---i.e., \eqref{eq:main} with $m\in(0,1)$--- global regularity in smooth bounded domains has been established through the work of Sacks \cite{Sacks},  DiBenedetto \cite{DiBenedetto}, Chen-DiBenedetto \cite{CDi}, Kwong \cite{KwongY, KwongY2},  DiBenedetto-Kwong-Vespri \cite{DKV}, and finally  the first and last authors \cite{JX19, JX22}. 
Instead of that the solutions of porous medium equation decay in time with the estimate \eqref{eq:stability}, the solutions to fast diffusion equations will extinct in finite time. The extinction profiles and their stability for  fast diffusion equations have been established by Berryman-Holland \cite{BH}, Kwong \cite{KwongY3}, Feireisl-Simondon \cite{FS}, Bonforte-Grillo-V\'azquez \cite{BGV}, Bonforte-Figalli \cite{BFig}, Akagi \cite{Akagi}, and Jin-Xiong \cite{JX19, JX20}. Higher order asymptotics was recently obtained by Choi-McCann-Seis \cite{CMS}.

\subsection{Acknowledgements}
 
T. Jin was partially supported by Hong Kong RGC grants GRF 16306320 and GRF 16303822, and NSFC grant 12122120.
X. Ros-Oton was partially supported by the European Research Council (ERC) under the Grant Agreement No 801867, the AEI project PID2021-125021NA-I00 (Spain), the MINECO grant RED2018-102650-T (Spain), and the Spanish State Research Agency, through the Mar\'ia de Maeztu Program for Centers and Units of Excellence in R\&D (CEX2020-001084-M).
J. Xiong was is partially supported by  the National Key R\&D Program of China No. 2020YFA0712900, and NSFC grants 11922104 and 11631002. All authors would like to thank Matteo Bonforte, Beomjun Choi and Juan Luis V\'azquez for interesting discussions and comments.

\subsection{Organization of the paper}

This paper is organized as follows. 
In Section \ref{sec:holdergradient}, we prove some H\"older estimates for a singular and degenerate nonlinear parabolic equation, which will lead to \eqref{eq:bootstrap-stp}. In Section \ref{sec:alltime}, we prove all time regularity of solutions to \eqref{eq:main} with compatible initial data. In Section \ref{sec:eventual}, we prove the optimal regularity after a waiting time for general initial data $u_0$, as well as fine asymptotics of the solution.

\section{\emph{A priori} H\"older estimates for a singular and degenerate nonlinear parabolic equation}\label{sec:holdergradient}

Let $G$ be a bounded function on $B_1^+$ satisfying  
\be\label{eq:G}
\frac{1}{\Lda} x_n \le G(x)\le \Lda x_n \quad \mbox{for }x\in B_1^+,
\ee
and $A(x)= (a_{ij}(x))_{n\times n}$ be a bounded matrix valued function in $B_1^+$ satisfying
\begin{equation}\label{eq:ellipticity}
\frac{1}{\Lda} |\xi|^2 \le \sum_{i,j=1}^n a_{ij}(x)\xi_i\xi_j \le \Lda |\xi|^2 \quad\forall\,\xi\in\R^n,\ x\in B_1^+, 
\end{equation} 
where $\Lda \ge 1$ is constant.  

Let
\begin{equation}\label{eq:rangeofp}
0<p\le 1.
\end{equation}
We consider positive bounded solutions of 
\begin{equation}\label{eq:main3}
G^{p+1} \partial_t w^p=\mbox{div}(A G^2\nabla w) \quad\mbox{in }B_1^+ \times (-1, 0]
\end{equation} 
satisfying 
\be\label{eq:lo-up}
\overline m\le w\le \overline M \quad \mbox{in }B_1^+  \times (-1, 0]\ \  \mbox{for some }0<\overline m\le \overline M<\infty.
\ee 
The case $p>1$ has been considered in \cite{JX22} by the first and third authors for proving the gradient H\"older estimates of solutions to fast diffusion equations. The range \eqref{eq:rangeofp} is for the purpose of the porous medium equation.

As in \cite{JX22}, the equation \eqref{eq:main3} will be understood in the sense of distribution, and we are interested in the \emph{a priori} H\"older estimates of its solutions that are Lipschitz continuous in $\overline B_1^+\times(-1,0]$. This Lipschitz continuity is assumed only for simplicity to avoid introducing more notations, and it is enough for our purpose. Since the  operator $\mbox{div}(A G^2\nabla\,\cdot )$ is very degenerate near the boundary $\pa B_1^+\cap\{x_n=0\}$, no boundary condition  should be  imposed there; see Keldys \cite{Keldys},  Oleinik-Radkevic \cite{OR} and the more recent paper Wang-Wang-Yin-Zhou \cite{WWYZ}.

The main result of this section is as follows. 

\begin{thm}\label{thm:holdernearboundary} 
Suppose \eqref{eq:G}, \eqref{eq:ellipticity} and \eqref{eq:rangeofp} hold. Suppose $w$ is Lipschitz continuous in $\overline B_1^+\times(-1,0]$, satisfies \eqref{eq:lo-up}, and is a solution of \eqref{eq:main3} in the sense of distribution. Then there exist $\gamma>0$ and $C>0$, both of which depend only on $n$, $p$, $\Lambda$, $\overline m$ and $\overline M$, such that
\[
|w(x,t)-w(y,s)|\le C (|x-y|+|t-s|)^\gamma\quad\forall\,(x,t), (y,s)\in \overline B_{1/2}^+\times(-1/4,0].
\]
\end{thm}

Throughout this section, we assume all the assumptions in Theorem \ref{thm:holdernearboundary}.  The proof of Theorem \ref{thm:holdernearboundary} will be similar to that of Theorem 3.1 in \cite{JX22} but has needs to be adapted to the case $p\in(0,1)$.
In particular, there are some notable differences, including the following:
\begin{itemize}
\item[(i).] The  Caccioppoli type inequality:  Lemma 3.2 in \cite{JX22} and Lemma \ref{lem:degiorgiclass} in the below.
\item[(ii).] The connection between the measures $\nu_2$ and $\nu_{p+1}$: see (15) in \cite{JX22} and \eqref{eq:connection} in the below.
\item[(iii).] The decay estimate of the distribution function of $w$ along the time: see Lemma 3.5 in \cite{JX22} and Lemma \ref{lem:decay-1} in the below.
\end{itemize}

Note that both Lemma 2.3 (the Sobolev inequality) and Proposition 2.4 (the  De Giorgi type isoperimetric inequality) in \cite{JX22}  holds for all $p>0$.

The  Caccioppoli type inequalities for \eqref{eq:main3} with \eqref{eq:rangeofp} reads as follows.

\begin{lem}\label{lem:degiorgiclass}
Let $k\in [\overline m, \overline M]$ and $\eta$ be a smooth function supported in $B_R(x_0)\times(-1,1)$, where $ B_R(x_0)\subset B_1$. Let $v=(w-k)^+$ and $\tilde v=(w-k)^-$.
Then, for every $-1<t_1\le t_2\le 0$ we have
\begin{equation}\label{eq:caccipoli1}
\begin{split}
&\sup_{t_1<t<t_2} \int_{B_R^+(x_0)} (v^{2}-Cv^3)\eta^2 G^{p+1} \,\ud x  +\int_{B_R^+(x_0)\times(t_1,t_2] } |\nabla (v\eta)|^2 G^2\,\ud x \ud t \\ 
&\le \int_{B_R^+(x_0)} v^{2}\eta^2 G^{p+1} \,\ud x\Big|_{t_1}  +  C \int_{B_R^+(x_0)\times(t_1,t_2] } \Big (|\nabla \eta|^2 G^2+ |\pa_t\eta|\eta G^{p+1}\Big) v^{2}\,\ud x\ud t,
\end{split}
\end{equation} 
and
\begin{equation}\label{eq:caccipoli2}
\begin{split}
&\sup_{t_1<t<t_2} \int_{B_R^+(x_0)} \tilde v^{2}\eta^2 G^{p+1} \,\ud x  +\int_{B_R^+(x_0)\times(t_1,t_2] } |\nabla (\tilde v\eta)|^2 G^2\,\ud x \ud t   \\ 
&\le \int_{B_R^+(x_0)} (\tilde v^{2}+C\tilde v^3)\eta^2 G^{p+1} \,\ud x\Big|_{t_1}   + C \int_{B_R^+(x_0)\times(t_1,t_2] } \Big (|\nabla \eta|^2 G^2+ |\pa_t\eta|\eta G^{p+1}\Big) \tilde v^{2}\,\ud x\ud t,
\end{split}
\end{equation} 
where $C>0$ depends only on $n$, $p$, $\Lambda$, $\overline m$ and $\overline M$.
\end{lem}
\begin{proof}
Since $0<p<1$, then we have
\be  \label{eq:nonlinear-term}
\begin{split}
\frac{1}{2}k^{p-1}v^2-C_0v^3&\le \frac{(v+k)^{p+1}}{p+1}-\frac{k(v+k)^p}{p}+\frac{k^{p+1}}{p(p+1)} \le \frac{1}{2}k^{p-1}v^2 \\
\frac{1}{2} k^{p-1}\tilde v^2&\le \frac{(k-\tilde v)^{p+1}}{p+1}-\frac{k(k-\tilde v)^p}{p}+\frac{k^{p+1}}{p(p+1)}\le \frac{1}{2} k^{p-1}\tilde v^2 +C_0 \tilde v^3,
\end{split}
\ee
for some $C_0>0$ depending only on $\overline m,\overline M, p$. With this change, the left proof will be the same as that of Lemma 3.1 of \cite{JX22}.
\end{proof}

For $x_0\in\partial \R^n_+$ and $R>0$, let 
$$Q_R(x_0,t_0):=B_R(x_0)  \times (t_0-R^{p+1}, t_0], \qquad \quad Q_R^+(x_0,t_0):=B_R^+(x_0)  \times (t_0-R^{p+1}, t_0].$$ We simply write them as $Q_R$ and $Q_R^+$ if $(x_0,t_0)=(0,0)$.

For $q>0$, let  $$\ud \mu_{q}= G^q \,\ud x,\quad \ud \nu_{q}= G^q \,\ud x\ud t$$  and 
\[
|A|_{\mu_q}= \int_{A} G^q \,\ud x\ \ \mbox{for }A\subset B_1^+, \qquad \quad  |\tilde A|_{\nu_q}=\int_{\tilde A} G^q \,\ud x\ud t\ \ \mbox{for }\widetilde A\subset Q_1^+.
\]
Then for  $\tilde A\subset Q_R^+$, since $0<p\le 1$ and $G$ satisfies \eqref{eq:G}, we have
\be \label{eq:connection}
\frac{|\tilde A|_{\nu_{2}}}{|Q_R^+|_{\nu_{2}}}\le   \frac{C R^{1-p}|\tilde A|_{\nu_{p+1}}}{R^{n+p+3}} = C \frac{|\tilde A|_{\nu_{p+1}}}{|Q_R^+|_{\nu_{p+1}}},
\ee
where $C = C(n,\Lambda,p)>0$ depends only on $n,p$ and $\Lambda$. As mentioned earlier, the inequality \eqref{eq:connection} is the second main change.

Given Lemma \ref{lem:degiorgiclass} and \eqref{eq:connection}, the proofs of the following lemma will be the same as that of Lemma 3.3 and Lemma 3.4 in \cite{JX22}. We omit the details.

\begin{lem}\label{lem:smallonlargeset} 
Let $0<R<1$ and
\[
\overline m\le m\le \inf_{Q_R^+} w\le \sup_{Q_R^+} w\le M\le \overline M.
\]
There exist  $0<\gamma_0<1$ and $0<\delta_0<1$, both of which depend only on $n$, $p$, $\Lambda$, $\overline m$ and $\overline M$, such that 

\begin{itemize}
\item[(i).] for  every $0<\delta \le \delta_0$, if 
\[
\frac{|\{(x,t)\in Q_R^+: w(x,t)>M-\delta\} |_{\nu_{p+1}}}{|Q_R^+|_{\nu_{p+1}}} \le \gamma_0,
\]
then
\[
w\le M-\frac{\delta}{2}\quad\mbox{in }Q_{R/2}^+.
\]

\item[(ii).] for  every $\delta >0$,  if 
\[
\frac{|\{(x,t)\in Q_R^+: w(x,t)<m+\delta\} |_{\nu_{p+1}}}{|Q_R^+|_{\nu_{p+1}}} \le \gamma_0,
\]
then
\[
w\ge m+\frac{\delta}{2}\quad\mbox{in }Q_{R/2}.
\]
\end{itemize}
\end{lem}

The last main change from the De Giorgi iteration for $p>1$ lies in  the following lemma.

\begin{lem}\label{lem:decay-1} Let $0<R<\frac12$, $0<a\le 1$, $-\frac12<t_0\le -aR^{p+1}$, $0<\sigma<1$, $\delta>0$, and $\overline M\ge M_a\ge \sup_{B_{2R}^+ \times [t_0, t_0+aR^{p+1}] } w$.  There exists $\delta_0>0$ depending only on $n$, $p$, $\Lambda$, $\overline m$ and $\overline M$ such that for  every $0<\delta \le \delta_0$, if 
\[
|\{x\in B_R^+: w(x,t)>M_a-\delta\}|_{\mu_{p+1}}\le (1-\sigma) |B_R^+|_{\mu_{p+1}} \quad \mbox{for any } t_0\le t\le  t_0+aR^{p+1} ,
\]
then 
\[
\frac{|\{(x,t)\in B_{R}^+ \times [t_0, t_0+aR^{p+1}] : w(x,t)>M_a-\frac{\delta}{2^\ell}\}|_{\nu_{p+1}}}{|B_{R}^+ \times [t_0, t_0+aR^{p+1}]|_{\nu_{p+1}}}\le \frac{C}{\sigma a^\frac{1}{1+p} \ell^{\frac{p}{1+p}}}\quad\forall\,\ell\in\mathbb{Z}^+,
\]
where $C$ depends only on $n$, $p$, $\Lambda$, $\overline m$ and $\overline M$.
\end{lem} 
\begin{proof} 
Let 
\[
A(k,R;t)= B_R^+ \cap \{w(\cdot, t)>k\}, \quad A(k,R)= B_{R}^+ \times [t_0, t_0+aR^{p+1}]  \cap \{w>k\}
\]
and 
\[
k_j= M_a- \frac{\delta}{2^j}. 
\]
By Proposition 2.4 of \cite{JX22}, we have 
\begin{align*}
&(k_{j+1}-k_j) |A(k_{j+1},R;t)|_{\mu_{p+1}} |B_R^+\setminus A(k_{j},R;t)|_{\mu_{p+1}}\\
& \le C R^{n+p+2} \left(\int_{A(k_{j},R;t) \setminus A(k_{j+1},R;t)} |\nabla w|^2 x_n^2\right)^{1/2} |A(k_{j},R;t) \setminus A(k_{j+1},R;t)|_{\mu_{2p}}^{1/2}\\
& \le C R^{n+p+2} \left(\int_{B_R^+} |\nabla  (w-k_j)^+|^2 x_n^2\right)^{1/2} |A(k_{j},R;t) \setminus A(k_{j+1},R;t)|_{\mu_{2p}}^{1/2}. 
\end{align*}
By the assumption, 
\[
|B_R^+\setminus A(k_{j},R;t)|_{\mu_{p+1}} \ge \sigma |B_R^+|_{\mu_{p+1}}= C(n,p) \sigma R^{n+p+1}. 
\]
By H\"older's inequality, we also have 
\[
|A(k_{j},R;t) \setminus A(k_{j+1},R;t)|_{\mu_{2p}}^{1/2} \le C R^{\frac{n(1-p)}{2(1+p)}} |A(k_{j},R;t) \setminus A(k_{j+1},R;t)|_{\mu_{p+1}}^{\frac{p}{1+p}} . 
\]
Integrating in the time variable, we have
\begin{align*}
&\int_{t_0}^{t_0+a R^{p+1}} |A(k_{j+1},R;t)|_{\mu_{p+1}} \,\ud t\\&
\le \frac{C2^{j+1}}{\delta \sigma} R^{1+\frac{n(1-p)}{2(1+p)}} \\
&\quad \cdot\int_{t_0}^{t_0+a R^{p+1}} \left[|A(k_{j},R) \setminus A(k_{j+1},R)|_{\mu_{p+1}}^{\frac{p}{1+p}}\left(\int_{B_R^+} |\nabla (w-k_j)^+|^2 x_n^2\,\ud x\right)^{1/2}  \right]\ud t\\
&\le \frac{C2^{j+1}}{\delta \sigma} R^{1+\frac{n(1-p)}{2(1+p)}+\frac{1-p}{2}}  \\
&\quad\cdot |A(k_{j},R) \setminus A(k_{j+1},R)|_{\nu_{p+1}}^{\frac{p}{1+p}} \left(\int_{B_R^+\times [t_0, t_0+aR^{p+1}]  } |\nabla (w-k_j)^+|^2 x_n^2\,\ud x \ud t\right)^{1/2},
\end{align*}
where we used H\"older's inequality in the second inequality. 

There exists $\delta_0>0$ depending only on $n$, $p$, $\Lambda$, $\overline M$ and $\overline m$ such that for $0<\delta<\delta_0$ and $ v= (w- k_j)^-$, there holds
\[
 v^2-C v^3\ge \frac{1}{2}  v^2,
\]
where the constant $C$ in the above is the one in \eqref{eq:caccipoli2}. Let $\eta(x) $ be a smooth cut-off function satisfying 
\begin{equation}\label{eq:testfunction}
\begin{split}
&\mbox{supp}(\eta) \subset B_{2R}, \quad 0\le \eta \le 1, \quad \eta=1 \mbox{ in }B_{R}, \quad |\nabla \eta(x)|^2  \le \frac{C(n)}{R^2} \quad \mbox{in }B_{2R}. 
\end{split}
\end{equation}
It follows from \eqref{eq:caccipoli1}  that
\begin{align*}
&\int_{t_0}^{t_0+a R^{p+1}}\int_{B_R^+  } |\nabla (w-k_j)^+|^2 x_n^2\,\ud x \ud t \\&
\le C \left(\int_{B_{2R}^+} |(w-k_j)^+(t_0)|^2 x_n^{p+1} \,\ud x+\frac{1}{R^2}\int_{t_0}^{t_0+a R^{p+1}}\int_{B_{2R}^+  } x_n^2 |(w-k_j)^+|^{2} \,\ud x\ud t \right )\\&
\le \frac{C \delta^2}{ 4^j} R^{n+p+1}. 
\end{align*}
Hence, 
\[
 |A(k_{j+1},R)|_{\nu_{p+1}} \le \frac{C}{\sigma} R^{\frac{n+2p+2}{p+1}} |A(k_{j},R) \setminus A(k_{j+1},R)|_{\nu_{p+1}}^{\frac{p}{1+p}}
\]
or 
\[
 |A(k_{j+1},R)|_{\nu_{p+1}}^{\frac{1+p}{p}} \le \frac{C}{\sigma^\frac{1+p}{p}} R^{\frac{n+2p+2}{p}} |A(k_{j},R) \setminus A(k_{j+1},R)|_{\nu_{p+1}}. 
\]
Taking a summation, we have 
\begin{align*}
\ell |A(k_{\ell},R)|_{\nu_{p+1}}^{\frac{1+p}{p}} \le \sum_{j=0}^{\ell-1}  |A(k_{j+1},R)|_{\nu_{p+1}}^{\frac{1+p}{p}} &\le  \frac{C}{\sigma^{\frac{1+p}{p}}} R^{\frac{n+2p+2}{p}} |B_{R}^+ \times [t_0, t_0+aR^{p+1}]|_{\nu_{p+1}} \\& \le  \frac{C}{a^\frac{1}{p}\sigma^\frac{1+p}{p}}  |B_{R}^+ \times [t_0, t_0+aR^{p+1}]|_{\nu_{p+1}}^\frac{1+p}{p}. 
\end{align*}
The lemma follows. 
\end{proof}

Similarly,

\begin{lem}\label{lem:decay-1'} Let $0<R<\frac12$,  $0<a\le 1$, $-\frac12<t_0\le -aR^{p+1}$,  $0<\sigma<1$, $\delta>0$ and $\overline m\le m_a\le \inf_{B_{2R}^+ \times [t_0, t_0+aR^{p+1}] } w$. If
\[
|\{x\in B_R^+: w(x,t)<m_a+\delta\}|_{\mu_{p+1}}\le (1-\sigma) |B_R^+|_{\mu_{p+1}} \quad \mbox{for any } t_0\le t\le  t_0+aR^{p+1} ,
\]
then 
\[
\frac{|\{(x,t)\in B_{R}^+ \times [t_0, t_0+aR^{p+1}] : w(x,t)<m_a+\frac{\delta}{2^\ell}\}|_{\nu_{p+1}}}{|B_{R}^+ \times [t_0, t_0+aR^{p+1}]|_{\nu_{p+1}}}\le \frac{C}{\sigma a^\frac{1}{1+p} \ell^{\frac{p}{1+p}}} \quad\forall\,\ell\in\mathbb{Z}^+,
\]
where $C$ depends only on $n$, $p$, $\Lambda$, $\overline m$ and $\overline M$.
\end{lem}  

Now we can estimate  the distribution function of $w$ at each time slice based on the starting time.

\begin{lem}\label{lem:decay-2} Let $0<R<\frac12$, $-\frac12<t_0\le -R^{p+1}$, $\overline M\ge M_1\ge \sup_{B_{2R}^+ \times [t_0, t_0+R^{p+1}] } w$ and $0<\sigma<1$. There exist constants $\delta_0>0$ and $s_0>1$ depending only on $n$, $p$, $\Lambda$, $\overline m$, $\overline M$ and $\sigma$  such that if $0<\delta<\delta_0$ and 
\[
|\{x\in B_R^+: w(x,t_0)>M_1-\delta\}|_{\mu_{p+1}}\le (1-\sigma) |B_R^+|_{\mu_{p+1}}, 
\]
then
\[
|\{x\in B_R^+: w(x,t)>M_1-\frac{\delta}{2^{s_0}}\}|_{\mu_{p+1}}\le (1-\frac{\sigma}{2}) |B_R^+|_{\mu_{p+1}} \quad \mbox{for every } t_0\le t\le  t_0+R^{p+1} . 
\] 
\end{lem}  

\begin{proof}  Let $\eta$ be a cut-off function supported in $B_R$ and $\eta=1$ in $B_{\beta R}$, where $0<\beta<1$ to be fixed.  Let $0<a\le 1$  and 
\[
A^a(k,R)= \{B_R^+ \times[t_0, t_0+a R^{p+1}] \}\cap \{w>k\}.
\]
Let $k_1>1$.  By Lemma \ref{lem:degiorgiclass},  we have
\begin{align*}
&\sup_{t_0<t< t_0+a R^{p+1}} \int_{B_R^+} (v^{2}-Cv^3)\eta^2 G^{p+1} \,\ud x  \\&
\quad\le \int_{B_R^+} v^{2}\eta^2 G^{p+1} \,\ud x\Big|_{t_0}  +  C \int_{B_R^+\times[t_0,t_0+a R^{p+1} ] } |\nabla \eta|^2 G^2 v^{2}\,\ud x\ud t,
\end{align*} 
where $v=(w- (M_1 -\delta))^+$. Choose $\delta_0$ small such that $1-C\delta_0>1/2$. Note that 
\begin{align*}
\int_{B_R^+} (v^{2}-Cv^3)\eta^2 G^{p+1} \,\ud x \Big|_t &\ge (1-C\delta)\delta^2(1-2^{-k_1})^2  |B_{\beta R} \cap \{w(x,t)> M_1- \delta 2^{-k_1}\}|_{\mu_{p+1}},\\
\int_{B_R^+} v^{2}\eta^2 G^{p+1} \,\ud x\Big|_{t_0}  &\le \delta^2 |\{x\in B_R^+: w(x,t_0)>M_1-\delta\}|_{\mu_{p+1}} \\&\le  \delta^2  (1-\sigma) |B_R^+|_{\mu_{p+1}},
\end{align*}
and 
\begin{align*}
 \int_{B_R^+\times[t_0,t_0+a R^{p+1} ] } |\nabla \eta|^2 G^2 v^{2}\,\ud x\ud t &\le \delta^2 \frac{C}{(1-\beta)^2R^2} |A^a(M_1- \delta ,R)|_{\nu_2}\\&
 \le \delta^2|B_R^+|_{\mu_{p+1}} \frac{C }{(1-\beta)^2} \frac{|A^a(M_1- \delta ,R)|_{\nu_2} }{|Q_R|_{\nu_2}}.
\end{align*}
It follows that for all $t\in [t_0, t_0+a R^{p+1}]$, 
\begin{align*}
& |B_{\beta R}^+ \cap \{w(x,t)> M_1- \delta 2^{-k_1}\}|_{\mu_{p+1}} \\
&\le |B_R^+|_{\mu_{p+1}}  \left( \frac{(1-\sigma)}{(1-C \delta) (1-2^{-k_1})^2 }+\frac{C }{(1-\beta)^2} \frac{|A^a(M_1- \delta ,R)|_{\nu_2} }{|Q_R|_{\nu_2}} \right)\\
&\le |B_R^+|_{\mu_{p+1}}  \left( \frac{(1+C \delta) (1-\sigma)}{(1-2^{-k_1})^2 }+\frac{C }{(1-\beta)^2} \frac{|A^a(M_1- \delta ,R)|_{\nu_2} }{|Q_R|_{\nu_2}} \right).
\end{align*}
Hence,
\begin{align*}
& |B_{R}^+ \cap \{w(x,t)> M_1- \delta 2^{-k_1}\}|_{\mu_{p+1}} \\&\le |B_R^+|_{\mu_{p+1}}  \left( C(1-\beta)+\frac{(1+C \delta) (1-\sigma)}{(1-2^{-k_1})^2 }+\frac{C }{(1-\beta)^2  } \frac{|A^a(M_1- \delta ,R)|_{\nu_2} }{|Q_R|_{\nu_2}} \right).
\end{align*}
By choosing $\beta$ such that
\[
(1-\beta)^3=\frac{|A^a(M_1- \delta ,R)|_{\nu_2} }{|Q_R|_{\nu_2}}, 
\] 
we have
\begin{align}\label{eq:smallinitiallater}
& |B_{R}^+ \cap \{w(x,t)> M_1- \delta 2^{-k_1}\}|_{\mu_{p+1}} \nonumber\\
&\le |B_R^+|_{\mu_{p+1}}  \left(\frac{(1+C \delta) (1-\sigma)}{(1-2^{-k_1})^2 }+C \Big(\frac{|A^a(M_1- \delta ,R)|_{\nu_2} }{|Q_R|_{\nu_2}}\Big)^{\frac13} \right),
\end{align}
where $C>0$ depends only on $n,p,\Lambda,\overline m$ and $\overline M$. 
Since
\[
\frac{|A^a(M_1- \delta ,R)|_{\nu_2} }{|Q_R|_{\nu_2}}\le a,
\]
we can choose $a$ small such that
\[
Ca^{1/3}\le\frac{\sigma}{8}.
\] 
Now we fix such an $a$. We choose $a$ slightly smaller if necessary to make $a^{-1}$ to be an integer. Let $N=a^{-1}$ and denote
 \[
 t_j=t_0+jaR^{p+1}\quad j=1,2,\cdots,N.
 \]
 We will inductively prove that there exist $s_1<s_2<\cdots<s_N$ such that
 \begin{equation}\label{eq:densitypropa}
\sup_{t_{j-1}\le t\le t_j}|B_{ R}^+ \cap \{w(x,t)> M_1- \delta 2^{-s_j}\}|_{\mu_{p+1}}  \le \left(1-\sigma+\frac{j}{4N} \sigma\right)  |B_R^+|_{\mu_{p+1}},
\end{equation}
where all the $s_j$ depend only on  $n,p,\Lambda,\overline m, \overline M$ and $\sigma$, from which the conclusion of this lemma follow.

Let us consider $j=1$ first.   

There exist $\delta_0$ small and $k_0$ large, both of which depends only on $n$, $p$, $\Lambda$, $\overline m$, $\overline M$ and $\sigma$,  such that  for all $\delta\in(0,\delta_0]$ and all $k_1\ge k_0$, we have
\[
\frac{(1+C \delta) (1-\sigma)}{(1-2^{-k_1})^2 }\le 1-\sigma+\frac{\sigma}{8N}.
\]
Then by \eqref{eq:smallinitiallater}, 
\[
|B_{ R}^+ \cap \{w(x,t)> M_1- \delta 2^{-k_1}\}|_{\mu_{p+1}}  \le \left(1-\frac{3}{4} \sigma\right)  |B_R^+|_{\mu_{p+1}} 
\]
 for all $t\in [t_0, t_1]$.  Applying Lemma \ref{lem:decay-1} and \eqref{eq:connection}, for every $k_2>k_1$, we have
\[
 \frac{|A^a(M_1- \delta 2^{-k_2},R)|_{\nu_2} }{|Q_R|_{\nu_2}}\le  C \frac{|A^a(M_1- \delta 2^{-k_2},R)|_{\nu_{p+1}} }{|Q_R|_{\nu_{p+1}}}  \le \frac{C}{\sigma}\left(\frac{a}{k_2-k_1}\right)^{\frac{p}{p+1}}.
\]
Hence, we can choose $k_2$ large enough such hat
\[
C\left( \frac{|A^a(M_1- \delta 2^{-k_2},R)|_{\nu_2} }{|Q_R|_{\nu_2}}\right)^{\frac 13}\le \frac{\sigma}{8N}.
\]
Let $k_1=k_0$ and $s_1=k_1+k_2$. By replacing $\delta$ by $\delta 2^{-k_2}$ in \eqref{eq:smallinitiallater}, it follows that
\[
\sup_{t_{0}\le t\le t_1}|B_{ R}^+ \cap \{w(x,t)> M_1- \delta 2^{-s_1}\}|_{\mu_{p+1}}  \le \left(1-\sigma+\frac{1}{4N} \sigma\right)  |B_R^+|_{\mu_{p+1}}.
\]
This prove \eqref{eq:densitypropa} for $j=1$. The proof for $j=2,3,\cdots,N$ is similar, by considering the starting time as $t_{j-1}$. We omit the details.
\end{proof}

 Similarly,

\begin{lem}\label{lem:decay-2'} Let  $0<R<\frac12$, $-\frac12<t_0\le -R^{p+1}$, $\overline m\le m_1\le \inf_{B_{2R}^+ \times [t_0, t_0+R^{p+1}] } u$ and $0<\sigma<1$. There exist constants $\delta_0>0$ and $k_0>1$ depending only on  $n$, $p$, $\Lambda$, $\overline m$, $\overline M$ and $\sigma$  such that if $0<\delta<\delta_0$ and 
\[
|\{x\in B_R^+: w(x,t_0)<m_1+\delta\}|_{\mu_{p+1}}\le (1-\sigma) |B_R^+|_{\mu_{p+1}}, 
\]
then
\[
|\{x\in B_R^+: w(x,t)<m_1+\frac{\delta}{2^{k_0}}\}|_{\mu_{p+1}}\le (1-\frac{\sigma}{2}) |B_R^+|_{\mu_{p+1}} \quad \mbox{for any } t_0\le t\le  t_0+R^{p+1} . 
\] 
\end{lem}  

Finally, combining all the above lemmas, we obtain the improvement of oscillation of $w$ at the boundary.

\begin{lem}\label{lem:decay-3} Let $0<R<\frac12$, $\overline M\ge M\ge \sup_{B_{2R}^+ \times [-R^{p+1},0] } u$ and $0<\sigma<1$. There exist constants $\delta_0>0$ and $k_0>1$ depending only on $n$, $p$, $\Lambda$, $\overline m$, $\overline M$ and $\sigma$  such that if $0<\delta<\delta_0$ and 
\[
|\{x\in B_R^+: w(x,-R^{p+1})>M-\delta\}|_{\mu_{p+1}}\le (1-\sigma) |B_R^+|_{\mu_{p+1}}, 
\]
then 
\[
\sup_{Q_{R/2}^+} w\le M-\frac{\delta}{2^{k_0}}.  
\]

\end{lem}

\begin{proof} It follows from Lemma \ref{lem:decay-2},  Lemma \ref{lem:decay-1} with $a=1$ and Lemma \ref{lem:smallonlargeset}.  
\end{proof}

\begin{lem}\label{lem:decay-3'} Let $0<R<\frac12$, $\overline m\le m\le \inf_{B_{2R}^+ \times [-R^{p+1},0] } u$ and $0<\sigma<1$. There exist constants $\delta_0>0$ and $k_0>1$ depending only on  $n$, $p$, $\Lambda$, $\overline m$, $\overline M$ and $\sigma$ such that if $0<\delta<\delta_0$ and 
\[
|\{x\in B_R^+: w(x,-R^{p+1})<m+\delta\}|_{\mu_{p+1}}\le (1-\sigma) |B_R^+|_{\mu_{p+1}}, 
\]
then 
\[
\inf_{Q_{R/2}^+} w\ge m+\frac{\delta}{2^{k_0}}.  
\]

\end{lem}

\begin{proof} It follows from Lemma \ref{lem:decay-2'},  Lemma \ref{lem:decay-1'} with $a=1$ and Lemma \ref{lem:smallonlargeset}.  
\end{proof}

Then the \emph{a priori} H\"older estimate at the boundary follows in a standard way.

\begin{thm}\label{thm:holderatboundary}
Suppose \eqref{eq:G}, \eqref{eq:ellipticity} and \eqref{eq:rangeofp} hold.  Suppose $w$ is Lipschitz continuous in $\overline B_1^+\times(-1,0]$, satisfies \eqref{eq:lo-up}, and is a solution of \eqref{eq:main3} in the sense of distribution. 
Then there exist $\alpha>0$ and $C>0$, both of which depend only on $n$, $p$, $\Lambda$, $\overline m$ and $\overline M$, such that for every $\bar x\in\pa' B_{1/2}^+$ and $\bar t\in (-1/4,0)$, there holds
\[
|w(x,t)-w(\bar x,\bar t)|\le C (|x-\bar x|+|t-\bar t|^{\frac{1}{p+1}})^\alpha\quad\forall\,(x,t)\in B_{1/2}^+(\bar x)\times(-1/4+\bar t,\bar t].
\]
\end{thm}

\begin{proof}[Proof of Theorem \ref{thm:holdernearboundary}]
Theorem \ref{thm:holdernearboundary} follows from Theorem \ref{thm:holderatboundary}, H\"older estimates for uniformly parabolic equations, and a scaling argument. It will be identical to that of Theorem 3.1 in \cite{JX22}, so that we omit the details.
\end{proof}

\section{All time regularity of solutions with compatible initial data}\label{sec:alltime}

 Let $\omega$ be a smooth function in $\overline\Omega$ comparable to the distance function $d$, that is, $0<\inf_{\Omega}\frac{\omega}{d}\le \sup_{\Omega}\frac{\omega}{d}<\infty$. 
 For example, $\omega$ can be taken as the nonnegative normalized first eigenfunction of $-\Delta$ in $\Omega$ with the zero Dirichlet boundary condition. Because of \eqref{eq:stability}, the linearized equation of \eqref{eq:main} will be of the form:
\begin{equation} \label{eq:general}
\begin{split}
\pa_t u-\omega(x)^{1-p} \left[\sum_{i,j=1}^n a_{ij}(x,t) u_{ij} + \sum_{i,j=1}^n b_i(x,t)u_i\right]+c (x,t) u  &=f \quad \mbox{in }\Omega \times(-1,0],
\end{split}
\end{equation}
where the matrix $(a_{ij}(x,t))_{n\times n}$ is  symmetric and satisfies
\be \label{eq:ellip}
\lda |\xi|^2 \le \sum_{i,j=1}^na_{ij}(x,t)\xi_i\xi_j\le \Lda |\xi|^2 \quad\forall\ \xi\in\R^n \mbox{ and } (x,t)\in \Omega \times(-1,0]
\ee
with $0<\lda\le \Lda<\infty$.

Let $\alpha\in (0,1)$ and 
\begingroup
\allowdisplaybreaks
\begin{align*}
[u]_{\mathscr{C}^\alpha( \om \times (0,T])}&:=\sup_{\substack{(x,t),(y,t)\in  \overline\om \times (0,T],\\ x\neq y}}\frac{|u(x,t)-u(y,t)|}{|x-y|^\frac{(1+p)\alpha}{2}} +\\
&\quad+ \sup_{\substack{(x,t),(y,t)\in  \overline\om \times (0,T],\\ d(x)>d(y)}}d(x)^{\frac{(1-p)\alpha}{2}}\frac{|u(x,t)-u(y,t)|}{|x-y|^\alpha} +\\
&\quad+\sup_{\substack{(x,t),(x,s)\in  \overline\om \times (0,T], \\ t\neq s}}\frac{|u(x,t)-u(x,s)|}{|t-s|^\frac{\alpha}{2}}.
\end{align*}
\endgroup
Denote
\begin{align*}
\|u\|_{\mathscr{C}^{\al}(\overline \om\times [0,T])}&= \|u\|_{L^\infty(\overline \om\times [0,T])}+[u]_{\mathscr{C}^{\al}(\overline \om\times [0,T])},\\
\|u\|_{\mathscr{C}^{2+\al}(\overline \om\times [0,T])}&=\|u\|_{\mathscr{C}^{\al}(\overline \om\times [0,T])}+\|u_t\|_{\mathscr{C}^{\al}(\overline \om\times [0,T])}+\|\nabla u\|_{\mathscr{C}^{\al}(\overline \om\times [0,T])}\\
&\quad+\|d^{1-p} D^2u\|_{\mathscr{C}^{\al}(\overline \om\times [0,T])}.
\end{align*}
Using the method of Daskalopoulos-Hamilton \cite{DH} proving Schauder estimates for \eqref{eq:general} with $p=0$, Kim-Lee \cite{KimLee} proved the following Schauder estimates for \eqref{eq:general} with all $p\in (0,1)$. Note that the above weighted H\"older norm is equivalent to those defined in  \cite{DH} and \cite{KimLee}, but just written in a different way.

\begin{thm}[\cite{KimLee}]\label{thm:global-schauder} 
Let $\Omega\subset\R^n$ be any bounded smooth domain.
Let $0<p<1$ and $0<\alpha<\min(\frac{2(1-p)}{1+p},1)$. Assume all $a_{ij}, b_i, c$ and $f$ belong to $\mathscr{C}^{\al}(\Omega \times[-1,0])$, \eqref{eq:ellip} holds, and $f(x,-1)=0$ for all $x\in\partial\Omega$.  Then there exists a unique classical solution $u$ of \eqref{eq:general} satisfying $u=0$ on $\partial_{pa}(\Omega \times(-1,0])$. Moreover,
\begin{align*}
\|u\|_{\mathscr{C}^{2+\al}(\overline \om\times [-1,0])}\le  C\|f\|_{\mathscr{C}^{\al} (\Omega \times [-1,0])},
\end{align*}
where  $C>0$ depends only on $n,p,\lda,  \Lda,\alpha,\Omega$, and the $\mathscr{C}^{\al} (\Omega\times[-1,0])$ norms of $a_{ij}, b_i$ and $c$. 
\end{thm}

We can also localize the above Schauder estimates in the time variable.

\begin{thm}\label{thm:globalxlocalt} 
Let $\Omega$, $p$, $\alpha$, $a_{ij}$, $b_i$, $c$, and $f$ as in Theorem \ref{thm:global-schauder}. 
Let $u\in \mathscr{C}^{2+\al}(\overline \om\times [-1,0])$  satisfy \eqref{eq:general} and $u=0$ on $\partial\Omega \times(-1,0]$. Then for every $q>0$, there exists $C>0$ depending only on $q, n,p,\lda,  \Lda,\alpha,\Omega$, and the $\mathscr{C}^{\al} (\Omega\times[-1,0])$ norms of $a_{ij}, b_i$ and $c$ such that
\begin{align*}
\|u\|_{\mathscr{C}^{2+\al}(\overline \om\times [-1/2,0])}\le  C(\|u\|_{L^q(\Omega \times [-1,0])}+\|f\|_{\mathscr{C}^{\al} (\Omega \times [-1,0])}).
\end{align*}
\end{thm}
\begin{proof}
Let $-1<\tau<s\le 0$. Let $\eta(t)$ be a cutoff function satisfying $\eta=0$ for $t \le \tau$, $\eta=1$ for $t\ge s$, $|\eta'(t)|\le C(s-\tau)^{-1}$ and $|\eta''(t)|\le C(s-\tau)^{-2}$, where $C$ is an absolute constant. 

Let $\tilde u=\eta u$. Then 
\[
\pa_t \tilde u-\omega(x)^{1-p} \left[\sum_{i,j=1}^n a_{ij}(x,t) \tilde u_{ij} + \sum_{i,j=1}^n b_i(x,t)\tilde u_i\right]+c (x,t) \tilde u  =\eta'(t) u+ f \quad \mbox{in }\Omega \times(-1,0].
\]
Then by Theorem \ref{thm:global-schauder}, we have
\[
\|u\|_{\mathscr{C}^{2+\al}(\overline \om\times [s,0])}\le  C(s-\tau)^{-2}\|u\|_{\mathscr{C}^{\al} (\Omega \times [\tau,0])}+C\|f\|_{\mathscr{C}^{\al} (\Omega \times [\tau,0])}.
\]
By an interpolation inequality, there exists $\beta>0$ depending only on $q$ such that 
\[
C(s-\tau)^{-2}\|u\|_{\mathscr{C}^{\al} (\Omega \times [\tau,0])}\le \frac{1}{2} \|u\|_{\mathscr{C}^{2+\al} (\Omega \times [\tau,0])} + C(s-\tau)^{-\beta}\|u\|_{L^q (\Omega \times [\tau,0])}.
\]
Hence, we have 
\[
\|u\|_{\mathscr{C}^{2+\al}(\overline \om\times [s,0])}\le  \frac{1}{2} \|u\|_{\mathscr{C}^{2+\al} (\Omega \times [\tau,0])} + C(s-\tau)^{-\beta}\|u\|_{L^q (\Omega \times [\tau,0])}+ C\|f\|_{\mathscr{C}^{\al} (\Omega \times [\tau,0])}.
\]
It follows from an iterative lemma, e.g., Lemma 1.1 in Giaquinta-Giusti \cite{GG}, that
\[
\|u\|_{\mathscr{C}^{2+\al}(\overline \om\times [s,0])}\le   C(s-\tau)^{-\beta}(\|u\|_{L^q (\Omega \times [\tau,0])}+ \|f\|_{\mathscr{C}^{\al} (\Omega \times [\tau,0])}),
\]
from which the conclusion follows.
\end{proof}

Consequently, they obtained a  short time existence result with compatible  initial data.

\begin{thm}[\cite{KimLee}]\label{thm:short-time} 
Let $p\in (0,1)$ and $v_0\in C^{1}(\overline \om)\cap C^{2}(\om)$ satisfy 
\begin{equation}\label{eq:essentialinitial}
\frac{1}{c_0}\leq \inf_{\Omega}\frac{v_0}{d}\le \sup_{\Omega}\frac{v_0}{d}\leq c_0
\end{equation}
for some constant $c_0>0$. Suppose 
\[d^{1-p} D^2 v_0\in C^\alpha(\Omega)\]
for some $\alpha>0$.  Then there exist a small $T>0$ and a unique positive function $v\in \mathscr{C}^{2+\al}(\overline \om\times [0,T])$ satisfying that
\[
pv^{p-1}\pa_t v =\Delta v \quad \mbox{in }\om\times [0,T],
\]
\[
v(\cdot,0)=v_0, \quad v=0 \quad \mbox{on }\pa \om \times [0,T].
\]
\end{thm}

The condition \eqref{eq:essentialinitial} is essential in obtaining this short time existence result.

Now we would like to bootstrap the $\mathscr{C}^{2+\al}$ solutions in Theorem \ref{thm:short-time} to reach their optimal regularity and obtain a desired uniform estimate  for them. The following lemma will be used.
\begin{lem}\label{lem:ellipticestimate}
Let $p\in(0,1)$. Suppose $u\in C^1(\overline\Omega)\cap C^2(\Omega)$ is a solution of 
\begin{align*}
-\Delta u(x)&=c(x) d(x)^{p-1}\quad\mbox{in }\Omega,\\
u(x)&=0\quad\mbox{on }\partial\Omega,
\end{align*}
where $c\in C^0(\overline\Omega)$. Then there exist $\alpha_0>0$ and $C>0$, both of which depend only on $n,p$ and $\Omega$, such that
\[
\|u\|_{C^{1+\alpha_0}(\overline\Omega)}\le C \|c\|_{L^\infty(\Omega)}.
\]
\end{lem}
\begin{proof}
It follows from the H\"older gradient estimates of the Green's functions near $\partial\Omega$ (see, e.g., Theorem 3.5 of Gr\"uter-Widman \cite{GW}) and elementary calculations.
\end{proof}

The estimate \eqref{eq:mainregularityestimate} in the below is the main contribution of this paper.
\begin{thm} \label{thm:bootstrap} 
Let $p\in (0,1)$, $\alpha\in (0,1)$ and $T>0$. Let  $v\in \mathscr{C}^{2+\al}(\overline \om\times [0,T])$ be a positive solution of
\begin{align*}
pv^{p-1}\pa_t v &=\Delta v \quad \mbox{in }\om\times [0,T]
\end{align*}
satisfying
\begin{align}\label{eq:lipassumption}
\frac{1}{c_0} d(x)\le v(x,t)\le c_0 d(x)\quad\mbox{in }\Omega\times[0,T]
\end{align}
for some constant $c_0>0$. Then $v(x,\cdot)\in C^\infty((0,T])$ for every $x\in\overline\Omega$, and
\[
\partial_t^\ell v(\cdot,t) \in C^{2+p}(\overline\Omega)\quad\mbox{for all }t\in(0,T]\mbox{ and }\ell\in\mathbb{Z}^+\cup\{0\}.
\]
Moreover, there exists $C>0$ depending only on $n,\Omega,T,p,\ell$ and $c_0$ such that
\begin{equation}\label{eq:mainregularityestimate}
\sup_{t\in[T/2,T]}\|\partial_t^\ell v(\cdot,t) \|_{C^{2+p}(\overline\Omega)}\le C.
\end{equation}
\end{thm}
\begin{proof}
Step 1. Since  $v\in \mathscr{C}^{2+\al}(\overline \om\times [0,T])$, we have
\[
\frac{v}{d}\in \mathscr{C}^{\al}(\overline \om\times [0,T]).
\]
Suppose
\be \label{eq:step0}
\left\|\frac{v}{d}\right\|_{\mathscr{C}^{\al}(\om \times [T/10, T])}\le M. 
\ee
By the Schauder estimates  in Theorem \ref{thm:globalxlocalt}, we have 
\[
\|v\|_{\mathscr{C}^{2+\al} (\overline\Omega \times [T/9,T])} \le C,
\]
where $C>0$ depends only on $n,p,\Omega,T, c_0, \alpha$ and the $M$ in \eqref{eq:step0}.

\medskip

Step 2. We claim that there exists $\gamma>0$ depending only on $n,p,\Omega$ and $\alpha$, and there exists $C>0$ depending only on $n,p,\Omega,T,c_0, \alpha$ and the $M$ in \eqref{eq:step0} such that
\begin{equation}\label{eq:vtregularity}
\|\partial_t v\|_{\mathscr{C}^{2+\gamma} (\overline\Omega \times [T/5,T])}\le C(n,p,\Omega,T,c_0,\alpha, M).
\end{equation}

Indeed, for $\lda\in (0,1]$ and arbitrarily  small positive number  $0<h<\frac{T}{100}$, we define 
\[
v^h_\lda(x,t)=\frac{v(x,t)-v(x,t-h)}{h^\lda}.
\]
By the equation of $v$, 
\be \label{eq:time-diff-quotient}
p d^{p-1} a \pa_t v_{\lda}^h =\Delta v_{\lda}^h +  d^{p-1}f \frac{v_{\lda}^h}{d}  \quad \mbox{in } \om \times (T/100,T],
\ee
where $a(x,t)=(\frac{v(x,t)}{d(x)})^{p-1}$ and
\begin{align*}
f(x,t)&= -d(x)^{2-p} \pa_t v(x,t-h) \int_0^1 \frac{\ud }{\ud s}[s v(x,t)+(1-s ) v(x,t-h)]^{p-1} \,\ud s     \\&
= -(p-1) \pa_t v(x,t-h)\int_0^1 \left[s \frac{v(x,t)}{d}+(1-s ) \frac{v(x,t-h)}{d}\right]^{p-2} \,\ud s.  
\end{align*}
By \eqref{eq:lipassumption}, $\frac{1}{c_0}\le a\le c_0$. This together with \eqref{eq:step0} implies that  \[
\|a\|_{\mathscr{C}^{\al} (\overline\Omega \times [T/8,T])} + \|f\|_{\mathscr{C}^{\al} (\overline\Omega \times [T/8,T])}\le C(n,p,\Omega,T,c_0,\alpha, M).
\]

Let $\alpha_0$ be the one in Lemma \ref{lem:ellipticestimate}. Let $\gamma=\min(\alpha,\alpha_0)$.
Set ${\lda_k}=\frac{(k+1)\gamma}{2}$ if $k<\frac{2}{\gamma}-1$ and $\lda_k=1$ if $k\ge \frac{2}{\gamma}-1$. By Step 1 and Taylor expansion calculations, we have 
\be \label{eq:dq-1}
|\pa_t v_{\lda_0}^h |+\left|\frac{v^h_{\lda_0}} d\right|\le C(n,p,\Omega,T,c_0,\alpha, M) \quad \mbox{in } \om\times [T/8, T]. 
\ee
Using \eqref{eq:dq-1} and applying Lemma \ref{lem:ellipticestimate} to \eqref{eq:time-diff-quotient} on each time slice, we have 
\[
\sup_{T/8\le t\le T}\|v_{\lda_0}^h (\cdot, t)\|_{C^{1,\gamma}(\om)} \le C(n,p,\Omega,T,\alpha, c_0,M).
\]
Combing with $|\pa_t v_{\lda_0}^h |\le C$ in \eqref{eq:dq-1}, it follows from a calculus lemma, Lemma 3.1 on page 78 in \cite{LSU} (cf. Lemma B.3 in \cite{JX19}), we have 
\[
|\nabla v_{\lda_0}^h(x,t) -\nabla v_{\lda_0}^h(y,s)| \le C(|x-y|^2+|t-s|)^{\frac{\gamma}{2}}, \quad \forall~ (x,t), (y,s)\in \om\times [T/8,T]. 
\]
Then we have
\[
\left\|\frac{v_{\lda_0}^h}{d}\right\|_{\mathscr{C}^{\gamma}(\overline\om\times [T/8,T])}\le C(n,p,\Omega,T,c_0,\alpha, M). 
\]
Applying the Schauder estimates in Theorem \ref{thm:globalxlocalt} to \eqref{eq:time-diff-quotient}, we then conclude that 
\be\label{eq:dq-2}
\|v_{\lda_0}^h\|_{\mathscr{C}^{2+\gamma} (\overline\Omega \times [T/7,T])}  \le C(n,p,\Omega,T,c_0,\alpha, M).
\ee
It follows that 
\[
|\pa_t v_{\lda_1}^h |+\left|\frac{v^h_{\lda_1}} d\right|\le C(n,p,\Omega,T,c_0,\alpha, M) \quad \mbox{in } \om\times [T/6, T].
\] 
Then one can repeat the above argument in finitely many steps to obtain \eqref{eq:vtregularity}. By applying elliptic Schauder estimates to the equation of $v$ on each time slice, we have that $D^2 v(\cdot,t)$ are H\"older continuous in $\overline\Omega$ for every $t\in [T/5,T]$. Consequently, $v/d$ is Lipschitz continuous on $\overline\Omega\times[T/5,T]$.
\medskip

Step 3: We show that there exist $\beta>0$ and $C>0$, both of which depend only on  $n$, $\om$, $T$, $p$ and $c_0$    such that
\be \label{eq:step1-a}
\left\|\frac{v}{d}\right\|_{\mathscr{C}^{\beta}(\om \times [T/4, T])}\le C(n,p,\Omega,T,c_0). 
\ee

Indeed, by \eqref{eq:lipassumption} and the H\"older regularity theory of linear uniformly parabolic equations, we only need to show the H\"older estimation of $v/d$ near the lateral boundary.  We  pick  a point  $x_0\in \om$ staying far away from the boundary $\pa \om$ and let  $G$ be the Green's function centered at $x_0$, i.e., 
\[
-\Delta G= \delta_{x_0} \quad \mbox{in }\om, \quad G=0 \quad \mbox{on }\pa \om. 
\] Then $\frac{1}{c_1}\le G/d\le c_1$ and $G/d$ is smooth in $\om_{\rho}:=\{x\in \om: d(x)<\rho\}$ for some constants $\rho< \frac{1}{2} d(x_0)$ and $c_1\ge 1$, both of which depend only on $\Omega$ and $n$.  Let
\[
w:=\frac{v}{G}.
\]
From Step 2, we know that $w$ is Lipschitz continuous on $\overline\Omega\times[T/5,T]$.
 Then it is elementary to check that 
\[
G^{p+1} \pa_t w^p =\mathrm{div}(G^2 \nabla w) \quad \mbox{in }\om_{\rho} \times [T/5,T]. 
\]
By straightening out the boundary $\partial\Omega$, and using the assumption \eqref{eq:lipassumption} and Theorem \ref{thm:holdernearboundary}, we have  
$$\| w\|_{\mathscr{C}^{\beta}(\om_{\rho/2} \times [T/4,T])} \le C$$  
for some $C>0$ depending only  $n$, $\om$, $T$, $p$, $b$, $c_0$ and $c_1$. Therefore, \eqref{eq:step1-a} follows. 

Step 4. Repeating steps 1-3 and replacing $\alpha$ by $\beta$, we can conclude that there exist $\beta>0$ and $C>0$, both of which depend only on $n,p,\Omega,T$ and $c_0$ such that
\[
\|v\|_{\mathscr{C}^{2+\beta} (\overline\Omega \times [T/3,T])}+\|\partial_t v\|_{\mathscr{C}^{2+\beta} (\overline\Omega \times [T/3,T])}\le C.
\]
By keeping differentiating the equation in the time variable, we have for every $\ell=1,2,\cdots$ that 
\[
\|\partial_t^\ell v\|_{\mathscr{C}^{2+\beta} (\overline\Omega \times [T/2,T])}\le C,
\]
where $C>0$ depends only on $n,p,\Omega,T, c_0$ and $\ell$. Using this estimate for $\partial_t^{\ell+2} v$ and applying the elliptic Schauder estimate to the equation of $\partial_t^{\ell+1} v$ on each time slice, we have for every $t\in[T/2,T]$ that 
\[
\|\partial_t^{\ell+1} v(\cdot,t)\|_{C^2(\overline\Omega)}\le C(n,p,\Omega,T, c_0,\ell).
\]
Applying again the elliptic Schauder estimate to the equation of $\partial_t^{\ell} v$ on each time slice, we have for every $t\in[T/2,T]$ that 
\[
\|\partial_t^{\ell} v(\cdot,t)\|_{C^{2+p}(\overline\Omega)}\le C(n,p,\Omega,T, c_0,\ell).
\]
Therefore, the proof is concluded.
\end{proof}

\begin{rem}\label{rem:c1alpha}
In the proof of Theorem \ref{thm:bootstrap}, we only used the estimate for  $\|u_t\|_{\mathscr{C}^{\al}(\overline \om\times [0,T])}+\|\nabla u\|_{\mathscr{C}^{\al}(\overline \om\times [0,T])}$ in Theorem \ref{thm:global-schauder} and Theorem \ref{thm:globalxlocalt}. The estimate for $\|d^{1-p} D^2u\|_{\mathscr{C}^{\al}(\overline \om\times [0,T])}$ is not used. 
\end{rem}

The regularity estimate \eqref{eq:mainregularityestimate} will imply the long time existence of regular solutions with compatible initial data.
\begin{thm} \label{thm:long-time0} 
Let $p\in (0,1)$ and $v_0\in C^{1}(\overline \om)\cap C^{2}(\om)$ satisfy \eqref{eq:essentialinitial} and $d^{1-p} D^2 v_0\in C^\alpha(\overline\Omega)$.  Then there exists a unique positive function $v\in \mathscr{C}^{2+\al}(\overline \om\times [0,+\infty))$ satisfying that
\[
pv^{p-1}\pa_t v =\Delta v \quad \mbox{in }\om\times [0,+\infty),
\]
\[
v(\cdot,0)=v_0, \quad v=0 \quad \mbox{on }\pa \om \times [0,+\infty),
\]
and  there exists $C>0$ depending only on $n,m,\Omega$ and $c_0$ such that
\begin{equation}\label{eq:decayinv}
\frac{1}{C} (1+t)^{\frac{1}{p-1}}S^{\frac{1}{p}}(x)\le v(x,t)\le C (1+t)^{\frac{1}{p-1}}S^\frac{1}{p}(x)\quad\mbox{on }\overline\Omega\times[0,+\infty),
\end{equation}
where $S$ is the unique solution of \eqref{eq:stationary} with $m=1/p$. Moreover, $v(x,\cdot)\in C^\infty((0,+\infty))$ for every $x\in\overline\Omega$, 
\[
\partial_t^\ell v(\cdot,t) \in C^{2+p}(\overline\Omega)\quad\mbox{for all }t\in(0,+\infty)\mbox{ and }\ell\in\mathbb{Z}^+\cup\{0\}.
\]
\end{thm}

\begin{proof}
The estimate \eqref{eq:decayinv} follows from the comparison principle. The long time existence and the regularity of the solution are then followed by repeatedly using Theorem \ref{thm:short-time} and Theorem \ref{thm:bootstrap}.
\end{proof}

For a solution $v$ satisfying the decay estimate \eqref{eq:decayinv},  we can obtain the decay estimates for its higher order regularity by a scaling argument.

\begin{thm} \label{thm:uniformestimate} 
Let $p\in (0,1)$, $\alpha\in (0,1)$ and    $v\in \mathscr{C}^{2+\al}(\overline \om\times [0,+\infty))$ be a positive solution of
\begin{align*}
pv^{p-1}\pa_t v &=\Delta v \quad \mbox{in }\om\times [0,+\infty)
\end{align*}
satisfying
\begin{align}\label{eq:lipassumption2}
\frac{1}{c_0} d(x)(1+t)^{\frac{1}{p-1}}\le v(x,t)\le c_0 d(x) (1+t)^{\frac{1}{p-1}}\quad\mbox{in }\Omega\times[0,+\infty)
\end{align}
for some $c_0>0$. Then for every $\delta>0$ and every $\ell\in\mathbb{Z}^+\cup\{0\}$, there exists $C>0$ depending only on $n,\Omega,p,\ell,\delta$ and $c_0$ such that
\begin{equation}\label{eq:higherintest}
\|\partial_t^\ell v(\cdot,t) \|_{C^{2+p}(\overline\Omega)}\le C (1+t)^{\frac{1}{p-1}-\ell}
\end{equation}
for all $t\in[\delta,+\infty)$.
\end{thm}
\begin{proof}
We only need to prove for large $t$. Arbitrarily fix a large $t_0>0$. Let 
\[
\tilde v(x,t)=t_0^{\frac{1}{1-p}} v(x,t_0t)\quad (x,t)\in\overline\Omega\times[1/2,1].
\]
Then
\begin{align*}
p\tilde v^{p-1}\pa_t \tilde v &=\Delta \tilde v \quad \mbox{in }\Omega\times[1/2,1].
\end{align*}
By the assumption \eqref{eq:lipassumption2}, we have
\[
\frac{1}{2^{\frac{1}{1-p}}c_0}d(x)\le \tilde v(x,t)\le 2^{\frac{1}{1-p}}c_0d(x)\quad\mbox{in }\Omega\times[1/2,1].
\]
By Theorem \ref{thm:bootstrap}, there exists $C>0$ depending only on $n,\Omega,p,\ell$ and $c_0$ such that 
\[
\|\partial_t^\ell \tilde v(\cdot,1) \|_{C^{2+p}(\overline\Omega)}\le C.
\]
That is,
\[
\|\partial_t^\ell v(\cdot,t_0) \|_{C^{2+p}(\overline\Omega)}\le C t_0^{\frac{1}{p-1}-\ell}.
\]
Since $t_0$ is arbitrarily, the conclusion is proved.
\end{proof}

\section{Eventual regularity for solutions with general initial data}\label{sec:eventual}

We now need an approximation argument to pass from the regularity for solutions with compatible initial data to those with general initial data. 

\begin{thm}\label{thm:withokinitialdata}
Let $u_0\in C^2(\Omega)$ be nonnegative such that $u_0^m$ a Lipschitz continuous function on $\overline\Omega$ satisfying
\[
C^{-1}d^{1/m} \leq u_0 \leq Cd^{1/m}\quad \textrm{in}\quad \Omega.
\]
Let $u$ be the weak solution of \eqref{eq:main}. Then 
 $u^m(x,\cdot)\in C^\infty((0,+\infty))$ for every $x\in\overline\Omega$, and
\[
\partial_t^\ell u^m(\cdot,t) \in C^{2+\frac{1}{m}}(\overline\Omega)\quad\mbox{for all }t\in(0,+\infty)\mbox{ and }\ell\in\mathbb{Z}^+\cup\{0\}.
\]
\end{thm}
\begin{proof}
Let $p=1/m$, 
\[
v_0=u_0^m\quad\mbox{and}\quad v=u^m.
\]
For every sufficiently small $\delta>0$, let $\eta_\delta\in C^\infty(\R)$ be such that $\eta_\delta\equiv 0$ on $(-\infty,\delta/2)$, $\eta_\delta\equiv 1$ on $(\delta,+\infty)$, $0\le\eta_\delta\le 1$ on $\R$ and $|\eta'_\delta|\le \frac{4}{\delta}$ on $[\frac{\delta}{2},\delta]$.  For $x\in\Omega$, set
\[
\phi_\delta(x)=\eta_\delta(d(x)).
\]
Then $\phi_\delta\equiv 0$ in $\{x\in\Omega: d(x)<\delta/2\}$, and $\phi_\delta(x)\equiv 1$ in $\Omega_\delta:=\{x\in\Omega: d(x)>\delta\}$. So we can extend $\phi_\delta$ to be identically zero in $\R^n\setminus\Omega$ such that $\phi_\delta\in C^\infty(\R^n)$. Let 
\[
v_{0,\delta}= \phi_\delta v_0 + (1-\phi_\delta) S^m,
\]
where $S$ is the unique solution of \eqref{eq:stationary}. It is elementary to check that there exist $C_0,c_0>0$ independent of $\delta$ such that
\[
\frac{1}{c_0}\le\inf_\Omega\frac{v_{0,\delta}}{d}\le \sup_\Omega\frac{v_{0,\delta}}{d}\le c_0,
\]
and
\[
|\nabla v_{0,\delta}|\le C_0.
\]
Note that $v_{0,\delta}=S^m$ near $\partial\Omega$. Hence, $v_{0,\delta}$ satisfies the assumptions of Theorem \ref{thm:long-time0}. Therefore, there exists a unique positive function $v_\delta\in \mathscr{C}^{2+\al}(\overline \om\times [0,+\infty))$ satisfying that
\[
pv^{p-1}_\delta \pa_t v_\delta =\Delta v_\delta \quad \mbox{in }\om\times [0,+\infty),
\]
\[
v_\delta(\cdot,0)=v_{0,\delta}, \quad v_\delta=0 \quad \mbox{on }\pa \om \times [0,+\infty),
\]
and \eqref{eq:decayinv} holds. Consequently, by Theorem \eqref{thm:uniformestimate}, for every $\va>0$, there exists $C_1>0$ depending only on $n,\Omega,m,\va$ and $c_0$ but independent of $\delta$ such that
\[
\|v_{\delta}\|_{\mathscr{C}^{2+\al}(\overline \om\times [\va,+\infty))}\le C_1.
\]
Since
\[
|v_{0,\delta}-v_0|\le C(1-\phi_\delta) d(x)\le C\delta,
\]
we have $v_{0,\delta}$ converges to $v_0$ uniformly on $\overline\Omega$ as $\delta\to 0$. Hence, by the Arzel\`a–Ascoli theorem and the uniqueness of the weak solutions, $v_\delta \to v$ locally uniformly on $\overline\om\times (0,+\infty)$. Hence, $v\in \mathscr{C}^{2+\al}(\overline \om\times [t_0,+\infty))$ for every $t_0>0$. The higher regularity then follows from Theorem~\ref{thm:long-time0}.
\end{proof}

Now we are ready to prove the main results of this paper.

\begin{proof}[Proof of Theorem \ref{eq:mainpme1}.]
By Theorem \ref{thm:withokinitialdata}, we have that $u^m(x,\cdot)\in C^\infty((T^*,+\infty))$ for every $x\in\overline\Omega$, and
\[
\partial_t^\ell u^m(\cdot,t) \in C^{2+\frac{1}{m}}(\overline\Omega)\quad\mbox{for all }t>T^*\mbox{ and }\ell\in\mathbb{Z}^+\cup\{0\}.
\] 
In particular, we can take $\ell=0$ to deduce \eqref{eq-optimal-x}.

The optimality of the exponent follows by recalling that the friendly giant solution \eqref{friendly-giant} does not belong to $C^{2+\frac1m+\varepsilon}(\overline\Omega)$ for any $\varepsilon>0$.
\end{proof}

\begin{proof}[Proof of Corollary \ref{cor-u}.]
Fix $t>T^*$.
Since both $u^m(\cdot,t)$ and $S^m(\cdot,t)$ are $C^{2+\frac1m}(\overline\Omega)$ and vanish linearly on the boundary $\partial\Omega$, it then follows that $u^m/S^m \in C^{1+\frac1m}(\overline\Omega)$, and that $u^m/S^m \asymp 1$ in~$\Omega$.
Thus, the regularity for $u/S$ follows by raising $u^m/S^m$ to the power $\frac1m$.

Finally, the optimality of the exponent follows from Example \ref{example:1} below.
\end{proof}

\begin{proof}[Proof of Theorem \ref{eq:mainpme}.]
Using Theorem \ref{thm:withokinitialdata}, the estimate \eqref{eq:regularityestimateu} follows from \eqref{eq:universalupperbound}, \eqref{eq:bounds} and Theorem \ref{thm:uniformestimate} with $v=u^m$ and $p=\frac{1}{m}$. In the following, we will prove \eqref{eq:decayestimateu} and \eqref{eq:decayestimateu2}.

Let
\begin{align*}
\theta(x,\tau)&=t^\frac{m}{m-1}u^m(x,t)\quad\mbox{with }t=e^\tau.
\end{align*}
Then
\begin{align*} 
\pa_\tau \theta^p &=\Delta \theta + \frac{p}{1-p} \theta^p\quad \mbox{in }\om \times (0,\infty).
\end{align*}
Consequently, the asymptotic expansion in \eqref{eq:stability} becomes
\begin{equation}\label{eq:stabilityinv}
\left\|\frac{\theta(\cdot,\tau)}{\Theta(\cdot)}-1\right\|_{L^\infty(\Omega)}\le C e^{-\tau}\quad\mbox{for all }\tau>1,
\end{equation}
where
\begin{equation}\label{eq:steadytheta}
\Theta=S^m
\end{equation}
satisfying
\begin{equation}\label{eq:steadythetaequation}
-\Delta \Theta - \frac{p}{1-p} \Theta^p=0 \quad\mbox{in }\Omega\quad \mbox{and}\quad \Theta=0\quad\mbox{on }\partial\Omega.
\end{equation}
Let
\[
h=\theta-\Theta.
\]
Then
\begin{align*}
p\theta^{p-1} h_\tau&=\Delta h + c(x,\tau) d(x)^{p-1}h \quad\mbox{in }\Omega\times(\log(T^*+1),\infty),\\
h&=0 \quad\mbox{on }\pa\Omega\times(\log(T^*+1),\infty),
\end{align*}
where
\[
c=\int_0^1 \left[\frac{s \theta+(1-s)\Theta}{d(x)}\right]^{p-1}\,\ud s.
\]
Note that the function $c$ shares the same regularity and the same estimates as those for $\theta$. 
By Theorem \ref{thm:globalxlocalt} and elliptic Schauder estimates on each time slice, there exists $C>0$ depending only on $n,\Omega,p$ and $u_0$ such that for all $T\ge \log T^*+2$, we have
 \begin{align*}
\sup_{\tau\in[T-1,T]}\|h(\cdot,\tau)\|_{C^{2+p}(\overline\om) }\le C\|h\|_{L^\infty(\overline \Omega\times[T-2,T])}\le C e^{-T},
\end{align*}
where we used \eqref{eq:stabilityinv} in the last inequality. In particular,  
\begin{align}\label{eq:finaldifferencek2}
\| \theta(\cdot,T)-\Theta\|_{C^{2+p}(\overline\om) }=\| h(\cdot,T)\|_{C^{2+p}(\overline\om) }\le  C e^{-T}.
\end{align}
Since $\theta(\cdot,\tau)=\Theta(\cdot)=0$ on $\pa\Omega\times[T^*,\infty)$, we have for all $\tau\ge \log T^*+2$ that
\begin{equation}\label{eq:relativeexpdecay}
\left\|\frac{\theta(\cdot,\tau)-\Theta}{\Theta}\right\|_{C^{1+p}(\overline \om)}\le C \|\theta(\cdot,\tau)-\Theta\|_{C^{2+p}(\overline \om)}=C \|h(\cdot,\tau)\|_{C^{2+p}(\overline \om)}\le C e^{-\tau}.
\end{equation}

Now let us obtain a higher order asymptotic expansion. This requires some spectrum analysis on the linearized equation of \eqref{eq:steadythetaequation} in a similar way to those in Bonforte-Figalli \cite{BFig} and Choi-McCann-Seis \cite{CMS} for the fast diffusion equation.  Let
\[
L^2(\Omega;\Theta^{p-1}\,\ud x)=\left\{f\in L^2(\Omega): \int_{\Omega}f^2(x)\Theta^{p-1}(x)\,\ud x<+\infty\right\}.
\] 
For $f,g\in L^2(\Omega;\Theta^{p-1}\,\ud x)$, we denote
\[
\langle f,g\rangle=\int_{\Omega}f(x)g(x)\Theta^{p-1}(x)\,\ud x\quad\mbox{and}\quad \|f\|=\sqrt{\langle f,f\rangle}.
\]
Let
\[
\mathcal{L}_\Theta = -\Delta -\frac{p^2}{1-p} \Theta^{p-1}
\]
be the linearized operator of \eqref{eq:steadythetaequation}. Since $\Theta(x)/d(x)$ is uniformly bounded from above and below by two positive constants and $p>0$, then the embedding $H^1_0(\Omega)\hookrightarrow L^2(\Omega;\Theta^{p-1}\,\ud x)$ is compact, and thus, $[\Theta^{1-p}(-\Delta)]^{-1}$ is a compact operator from $L^2(\Omega;\Theta^{p-1}\,\ud x)$ to itself. Therefore, the weighted eigenvalue problem
\be \label{eigenproblem}
\left\{\begin{array}{rcll}
\mathcal{L}_{\Theta} (\psi)  &=& \mu \Theta^{p-1}\psi   \  & \mbox{on } \Omega,\\
\Theta&=&0 \   &\mbox{on } \partial\Omega
\end{array}\right.
\ee
admits eigenpairs $\{ (\mu_j,\psi_j)       \}_{j=1}^{\infty}$ such that

\begin{itemize}
           \item  the eigenvalues with multiplicities can be listed as $\mu_1<\mu_2\le \mu_3 \le \cdots \le \mu_j \to +\infty$ as  $j \to +\infty$,
           \item the eigenfunctions $\{ \psi_j\}_{j=1}^{\infty}$ form a complete orthonormal basis of  $L^2(\Omega; \Theta^{p-1}\,\ud x) $, that is, $\langle \psi_i, \psi_j \rangle =  \delta_{ij}$ for $i,j \in \mathbb{N} $. Moreover, $\psi_1$ does not change signs, and thus, we assume that  $\psi_1\ge 0$.
\end{itemize}
Since $\Theta$ satisfies \eqref{eq:steadythetaequation}, then 
\begin{equation}\label{eq:firsteigen}
\mu_1=p\quad\mbox{and}\quad \psi_1=\frac{\Theta}{\| \Theta  \|}=\frac{\Theta}{\| \Theta  \|_{L^{p+1}(\Omega)}^{\frac{p+1}{2}}}. 
\end{equation}
The equation of $h$ can be rewritten as
\begin{equation}\label{eq:erroreq}
\begin{split}
p\Theta^{p-1} h_\tau&=-\mathcal{L}_\Theta h + N(h) \quad\mbox{in }\Omega\times[T_0,\infty),\\
h&=0 \quad\mbox{on }\pa\Omega\times[T_0,\infty),
\end{split}
\end{equation}
where
\[
N(h)=\frac{p}{1-p} \Theta^{p} \left[\left(\frac{h}{\Theta}+1\right)^{p}- 1-\frac{ph}{\Theta}\right] + p\Theta^{p-1}\left[1-\left(\frac{h}{\Theta}+1\right)^{p-1}\right] h_\tau.
\]
It follows from \eqref{eq:higherintest} and \eqref{eq:relativeexpdecay} that for all $\tau\ge\tau_0$, where $\tau_0$ is sufficiently large, we have
\begin{equation}\label{eq:quadraticerror}
|N(h)|\le C \Theta^{p-2}(h^2+h h_\tau)\le C\Theta^{p} e^{-2\tau} . 
\end{equation}
Therefore, for every $f\in L^2(\Omega;\Theta^{p-1}\,\ud x)$, we have
\[
\int_{\Omega} |N(h(x,\tau))| |f(x)|\,\ud x\le  C e^{-2\tau} \|f\|.
\]
Since $\mu_1=p$, we can define $J$ to be the largest positive integer such that $\mu_J<2p\le \mu_{J+1}$. Multiplying $\psi_j$, $j=1,\cdots,J$, to \eqref{eq:erroreq} and integrating by parts, we obtain
\[
\left|\frac{\ud}{\ud\tau}\langle h, \psi_j \rangle + \frac{\mu_j}{p} \langle h, \psi_j \rangle\right| \le C e^{-2\tau}.
\]
That is,
\[
\left|\frac{\ud}{\ud\tau}\left( e^{\frac{\mu_j}{p}\tau}\langle h, \psi_j \rangle \right)\right| \le C e^{-(2-\frac{\mu_j}{p})\tau}.
\]
Hence, for any $\tau_2>\tau_1\ge \tau_0$, we have
\[
\left|e^{\frac{\mu_j}{p}\tau_2}\langle h(\cdot,\tau_2), \psi_j \rangle -e^{\frac{\mu_j}{p}\tau_1}\langle h(\cdot,\tau_1), \psi_j \rangle\right|\le C e^{-(2-\frac{\mu_j}{p})\tau_1}.
\]
Therefore, the limit $\lim_{\tau\to\infty}e^{\frac{\mu_j}{p}\tau}\langle h(\cdot,\tau), \psi_j \rangle$ exists, and we denote it as $c_j$. Hence,
\[
\langle h(\cdot,\tau), \psi_j \rangle= c_j e^{-\frac{\mu_j}{p}\tau} + O(e^{-2\tau}).
\]
Let $z(\tau)=\|h-\sum_{j=1}^J \langle h(\cdot,\tau), \psi_j \rangle \psi_j\|$. Then we obtain a similar but just one side inequality from \eqref{eq:erroreq}:
\[
\frac{\ud}{\ud\tau} z(\tau)+\frac{\mu_{J+1}}{p} z(\tau) \le Ce^{-2\tau}.
\]
That is
\[
\frac{\ud}{\ud\tau} \left(e^{\frac{\mu_{J+1}}{p} }z(\tau) \right)\le Ce^{-(2-\frac{\mu_{J+1}}{p} )\tau}.
\]
Hence, for all $\tau\ge\tau_0$, we have
\begin{equation*}
z(\tau)\le
\begin{cases}
Ce^{-2  \tau} & \text { if } \mu_{J+1}>2p , \\ 
C\tau e^{-2 \tau} & \text { if } \mu_{J+1}= 2p.
\end{cases} 
\end{equation*}
Therefore, we have
\begin{equation}\label{eq:higherexpansion}
\left\|h(\cdot,\tau)-\sum_{j=1}^J c_j e^{-\frac{\mu_j}{p}\tau}\psi_j\right\|\le
\begin{cases}
e^{-2  \tau} & \text { if } \mu_{J+1}>2p , \\ 
\tau e^{-2 \tau} & \text { if } \mu_{J+1}= 2p.
\end{cases} 
\end{equation}
In particular, 
\begin{equation}\label{eq:secondexpansion}
\left\|h(\cdot,\tau)- c_1 e^{-\frac{\mu_1}{p}\tau}\psi_1\right\|\le
\begin{cases}
Ce^{-\frac{\mu_2}{p}\tau}+Ce^{-2  \tau} & \text { if } \mu_{J+1}>2p , \\ 
Ce^{-\frac{\mu_2}{p}\tau}+C\tau e^{-2 \tau} & \text { if } \mu_{J+1}= 2p.
\end{cases} 
\end{equation}
By using \eqref{eq:firsteigen}, we know that there exists $\gamma>0$ such that
\[
\left\|\theta(\cdot,\tau)-\Theta+ A_1 e^{-\tau}\Theta\right\|\le Ce^{-(1+\gamma)\tau},
\]
where
\[
A_1=-c_1 \| \Theta  \|^{-1}=-c_1 \| \Theta  \|_{L^{p+1}(\Omega)}^{-\frac{p+1}{2}}=-c_1 \|S\|_{L^{m+1}(\Omega)}^{-\frac{m+1}{2}}.
\]
By \eqref{eq:universalupperbound}, we know that $h\le 0$. Hence, $c_1\le 0$, and thus, $A_1\ge 0$.
If we let 
\[
\widetilde h=\theta(\cdot,\tau)-\Theta + A_1 e^{-\tau}\Theta,
\]
then 
\begin{equation}\label{eq:erroreq2}
\begin{split}
p\Theta^{p-1} \widetilde h_\tau&=-\mathcal{L}_\Theta \widetilde h + N(h) \quad\mbox{in }\Omega\times[T_0,\infty),\\
\widetilde h&=0 \quad\mbox{on }\pa\Omega\times[T_0,\infty).
\end{split}
\end{equation}
Since
\[
\|\pa_\tau N(h(\cdot,\tau))\|_{C^p(\overline\Omega)}\le Ce^{-2\tau},
\]
we obtain from Theorem \ref{thm:globalxlocalt} and H\"older's inequality that
\[
\|\widetilde h(\cdot,\tau)\|_{C^{2+p}(\overline\Omega)}\le C \|\widetilde h(\cdot,\tau)\|+Ce^{-2\tau}\le Ce^{-(1+\gamma)\tau}.
\]
Then \eqref{eq:decayestimateu} follows by changing the variables back to $u$ and $t$, and \eqref{eq:decayestimateu2} follows from \eqref{eq:decayestimateu}. 

Finally, in the dimension $n=1$ case, as Berryman \cite{Berryman} observed, $\mu_2=\frac{3p}{1-p}$ with the eigenfunction $\Theta\Theta'$. Hence, $J=1$. Therefore, it follows from \eqref{eq:secondexpansion} that $\gamma=1$.

This finishes the proof of Theorem~\ref{eq:mainpme}.
\end{proof}

\begin{rem} \label{rem:4.2}
In fact, we can keep expanding the solution up to an arbitrary order in the same way as that Han-Li-Li \cite{HLL} did for the singular Yamabe equation. If $J\ge 2$, then $N(h)=C_1e^{-2\tau} \Theta^p +O(e^{-\frac{(\mu_1+\mu_2)\tau}{p}})\Theta^p =:N_1(h)+N_2(h)$, where $C_1$ is a constant. Then one can solve the linear equation \eqref{eq:erroreq}  with the forcing term $N(h)$ replaced by $N_1(h)$ and $N_2(h)$, respectively. The equation with $N_1(h)$ has an explicit solution. Since $N_2(h)=O(e^{-\frac{(\mu_1+\mu_2)\tau}{p}})\Theta^p$, we can expand its solution up to the largest $K$ such that $\mu_{K}<\mu_1+\mu_2$. If $J=1$, then $N(h)=N_1(h)+N_2(h)$ with the same $N_1(h)$, and $N_2(h)=O(e^{-3\tau})\Theta^p $ or $O(\tau e^{-3\tau})\Theta^p $. In each case, $N_2(h)$ has strictly better decay than $N(h)$, and  one can expand the solutions up to the largest $K$ such that $\mu_{K}<3$. One can keep expanding $N(h)$ and iterating this process to reach any desired order, which is ensured by the regularity of the solution $\theta$. In the final expansion, the exponential exponents are not only the $\{\mu_j/p\}$ but also some of their linear combinations. In 1-D, an arbitrarily high order expansion for a special class of solutions under self-similar type coordinates has been obtained by Angenent \cite{Angenent2}.
\end{rem}

In the next example in one spatial dimension, we will show that the $C^{1+\frac{1}{m}}(\overline\Omega)$ regularity of the relative error $\frac{u^m(\cdot,t)}{S^m}$ cannot be improved in a short time after $T^*$.

\begin{example}\label{example:1}
Consider the equations \eqref{eq:main} and \eqref{eq:stationary} for $\Omega=(0,1)\subset\R$. By the regularity of $S$, we have the expansion of $S^m$ near $x=0$:
\[
S^m=ax-\frac{a^{\frac{1}{m}}m^2}{(m-1)(m+1)(2m+1)} x^{2+\frac{1}{m}}+\mbox{higher order terms}
\]
for some constant $a>0$.

Let $v_0$ be a function such that 
\begin{equation*}
v_0(x)=\left\{
\begin{array}{rcll}
&ax+\frac{a^{\frac{1}{m}}m^2}{(m-1)(m+1)(2m+1)} x^{2+\frac{1}{m}}&\quad\mbox{for }x\in [0,1/4],\\
&a(1-x)+\frac{a^{\frac{1}{m}}m^2}{(m-1)(m+1)(2m+1)} (1-x)^{2+\frac{1}{m}}&\quad\mbox{for }x\in [3/4,1],
\end{array}
\right.
\end{equation*}
and $v_0$ is smooth and positive in $(1/8,7/8)$. Let $d(x)=\min(x,1-x)$ for $x\in [0,1/4]\cup [3/4,1]$, and $d(x)$ be smooth and positive in $(1/8,7/8)$. Then both $d^{1-p} (v_0)''$ and $d^{1-p} [d^{1-p} (v_0)'']''$ are H\"older continuous on $[0,1]$. Then by the Schauder estimate in Theorem \ref{thm:global-schauder} and the implicit function theorem, one can show that there exists $T>0$ and a unique positive function $v\in \mathscr{C}^{2+\al}([0,1] \times [0,T])$ satisfying that $v_t\in \mathscr{C}^{2+\al}([0,1] \times [0,T])$ and
\[
pv^{p-1}\pa_t v =\Delta v \quad \mbox{in }[0,1]\times [0,T],
\]
\[
v(\cdot,0)=v_0, \quad v=0 \quad \mbox{on } \{0,1\} \times [0,T].
\]
This solution $v$ is more regular in the time variable than the one obtained by Theorem \ref{thm:short-time}. This is achievable because the initial condition $v_0$ is more regular. In fact, such a solution $v$ can be obtained by a second approximation (in the time variable) in a similar way to that in Theorem 3.2 of \cite{JX19}.

Let $u=v^{1/m}$. Then by the regularity of $u^m$ and $(u^m)_t$, we have near $x=0$ that
\begin{align}
u^m &= A(t)x+O(x^{1+\alpha}),\label{eq:expansion1}\\
  (u^m)_t &= B(t)x+O(x^{1+\alpha})\nonumber
\end{align}
for some positive continuous functions $A(t)$ and $B(t)$ on $[0,1]$ with $A(0)=a$ and $B(0)=\frac{am}{m-1}$, and some $\alpha>0$, where $O(x^{1+\alpha})$ is uniform on $[0,T]$.
Then
\begin{align*}
 u_t &= \frac{1}{m} u^{1-m} (u^m)_t \\
         &= \frac{1}{m} \left(A(t)x+O(x^{1+\alpha})\right)^{\frac{1-m}{m}}(B(t)x+O(x^{1+\alpha}))\\
         &= \frac{1}{m}A(t)^{\frac{1-m}{m}} B(t) x^{\frac{1}{m}} (1+O(x^\alpha))^{\frac{1-m}{m}}(1+O(x^\alpha))\\
         &= \frac{1}{m}A(t)^{\frac{1-m}{m}} B(t)x^{\frac{1}{m}} (1+ O(x^\alpha)).
\end{align*}
Since $(u^m)_{xx}=u_t$,  then integrating in $x$ and using \eqref{eq:expansion1}, we have that 
\[
(u^m)_{x}(x,t)- (u^m)_{x}(0,t)=\frac{1}{m}A(t)^{\frac{1-m}{m}} B(t)x^{1+\frac{1}{m}} (1+ O(x^\alpha)).
\]
That is,
\[
(u^m)_{x}(x,t)=A(t)+\frac{1}{m+1}A(t)^{\frac{1-m}{m}} B(t)x^{1+\frac{1}{m}} (1+ O(x^\alpha)).
\]
Integrating in $x$ again, we have
\[
(u^m)(x,t)=A(t)x+\frac{m}{(m+1)(2m+1)}A(t)^{\frac{1-m}{m}} B(t)x^{2+\frac{1}{m}} (1+ O(x^\alpha)).
\]
Hence, near $x=0$, we have
\[
\frac{u^m}{S^m}=\frac{A(t)}{a} \left(1+ f(t) x^{1+\frac{1}{m}}+\mbox{higher order terms}\right),
\]
where
\[
f(t)=\frac{m}{(m+1)(2m+1)}A(t)^{\frac{1-2m}{m}} B(t) + \frac{a^{\frac{1-m}{m}}m^2}{(m-1)(m+1)(2m+1)}.
\]
Since $f(t)$ is continuous on $[0,T]$ and
\[
f(0)=\frac{2a^{\frac{1-m}{m}}m^2}{(m-1)(m+1)(2m+1)}>0,
\]
we have $f(t)>0$ on $[0,\widetilde T]$ for some small $\widetilde T>0$. Hence, $\frac{u^m(\cdot,t)}{S^m}$ is precisely $C^{1+\frac{1}{m}}([0,1])$ for $t\in[0,\widetilde T]$.
\end{example}

\bigskip

\noindent T. Jin

\noindent Department of Mathematics, The Hong Kong University of Science and Technology\\
Clear Water Bay, Kowloon, Hong Kong\\[1mm]
Email: \textsf{tianlingjin@ust.hk}

\medskip

\noindent X. Ros-Oton

\noindent ICREA, Pg. Llu\'is Companys 23, 08010 Barcelona, Spain \& Universitat de Barcelona, Departament de Matem\`atiques i Inform\`atica, Gran Via de les Corts Catalanes 585, 08007 Barcelona, Spain \& Centre de Recerca Matem\`atica, Barcelona, Spain.\\[1mm]
Email: \textsf{xros@icrea.cat}

\medskip

\noindent J. Xiong

\noindent School of Mathematical Sciences, Laboratory of Mathematics and Complex Systems, MOE\\ Beijing Normal University, 
Beijing 100875, China\\[1mm]
Email: \textsf{jx@bnu.edu.cn}


\small

\begin{thebibliography}{99}

\bibitem{Akagi} G. Akagi,
             \emph{Rates of convergence to non-degenerate asymptotic profiles for fast diffusion via energy methods.}
             arXiv:2109.03960. 
             
\bibitem{Angenent} S. Angenent,
\emph{Analyticity of the interface of the porous media equation after the waiting time.}
Proc. Amer. Math. Soc. \textbf{102} (1988), no. 2, 329--336.

\bibitem{Angenent2} S. Angenent,
\emph{Large time asymptotics for the porous media equation.} Nonlinear diffusion equations and their equilibrium states, I (Berkeley, CA, 1986), 21--34,
Math. Sci. Res. Inst. Publ., 12, Springer, New York, 1988.

\bibitem{A} D. G. Aronson, 
            \emph{Regularity propeties of flows through porous media.} 
            SIAM J. Appl. Math. \textbf{17} (1969), 461--467.

\bibitem{A70}
            D. G. Aronson, 
            \emph{Regularity properties of flows through porous media: The interface.} 
            Arch. Rational Mech. Anal. \textbf{37} (1970), 1--10.
            
\bibitem{AB} D. G. Aronson and P. B\'enilan, 
         \emph{R\'egularit\'e des solutions de l'\'equation des milieux poreux dans $\mathbb{R}^n$.} (French) 
         C. R. Acad. Sci. Paris S\'er. A-B \textbf{288} (1979), no. 2, A103--A105.

\bibitem{ACV} D. G. Aronson, L. A. Caffarelli and J. L. V\'azquez, 
         \emph{Interfaces with a corner point in one-dimensional porous medium flow}.
Comm. Pure Appl. Math. \textbf{38} (1985), no. 4, 375--404.

\bibitem{AP1981} D. G. Aronson and L. A. Peletier, 
         \emph{Large time behaviour of solutions of the porous medium equation in bounded domains}.
J. Differential Equations \textbf{39} (1981), no. 3, 378--412.

\bibitem{AV87} D. G. Aronson and J. L. V\'azquez,
         \emph{Eventual $C^\infty$-regularity and concavity for flows in one-dimensional porous media.}
Arch. Rational Mech. Anal. \textbf{99} (1987), no. 4, 329--348.

\bibitem{AGS}  B. Avelin, U. Gianazza and S. Salsa, 
           \emph{Boundary estimates for certain degenerate and singular parabolic equations.} 
           J. Eur. Math. Soc. (JEMS) \textbf{18} (2016), no. 2, 381--424.

\bibitem{Berryman} J. G. Berryman,  
      \emph{Evolution of a stable profile for a class of nonlinear diffusion equations with fixed boundaries}.
          J. Math. Phys. \textbf{18} (1982), no. 11, 2108–2115.
           
 \bibitem{BH} J. G. Berryman and C. J. Holland,
            \emph{Stability of the separable solution for fast diffusion.}
            Arch. Rational Mech. Anal. \textbf{74} (1980), 379--388.


            
\bibitem{BFig} M. Bonforte and A. Figalli, 
           \emph{Sharp extinction rates for fast diffusion equations on generic bounded domains.} 
           Comm. Pure Appl. Math. \textbf{74} (2021), no. 4, 744--789. 
           
\bibitem{BFR17} M. Bonforte, A. Figalli, X. Ros-Oton, \emph{Infinite speed of propagation and regularity of solutions to the fractional porous medium equation in general domains}. Comm. Pure Appl. Math. \textbf{70} (2017), 1472--1508.

\bibitem{BFV18} M. Bonforte, A. Figalli, J. L. V\'azquez, \emph{Sharp global estimates for local and nonlocal porous medium-type equations in bounded domains}. Anal. PDE \textbf{11} (2018), 945--982.
           
           
\bibitem{BGV} M. Bonforte, G. Grillo and J. L. V\'azquez, 
          \emph{Behaviour near extinction for the fast diffusion equation on bounded domains.} 
          J. Math. Pures Appl. \textbf{97} (2012), 1--38. 

\bibitem{BSV15} M. Bonforte, Y. Sire and J. L. V\'azquez,  
          \emph{Existence, uniqueness and asymptotic behaviour for fractional porous medium equations on bounded domains}. Discrete Contin. Dyn. Syst. \textbf{35} (2015), no. 12, 5725--5767.

\bibitem{BV15} M. Bonforte and J. L. V\'azquez, 
          \emph{A priori estimates for fractional nonlinear degenerate diffusion equations on bounded domains.} Arch. Ration. Mech. Anal. \textbf{218} (2015), no. 1, 317--362.

          
\bibitem{CF} L. A. Caffarelli and A. Friedman,
          \emph{Regularity of the free boundary of a gas flow in an n-dimensional porous medium.}
Indiana Univ. Math. J. \textbf{29} (1980), no. 3, 361--391.

\bibitem{CVW} L. A. Caffarelli, J. L. V\'azquez and  N. I. Wolanski,
          \emph{Lipschitz continuity of solutions and interfaces of the N-dimensional porous medium equation.}
Indiana Univ. Math. J. \textbf{36} (1987), no. 2, 373--401.

\bibitem{CW} L. A. Caffarelli and N. Wolanski, 
          \emph{$C^{1,\alpha}$ regularity of the free boundary for the N-dimensional porous media equation.}
Comm. Pure Appl. Math. \textbf{43} (1990), no. 7, 885--902.

 \bibitem{CDi} Y. Z. Chen and E. DiBenedetto,
         \emph{On the local behavior of solutions of singular parabolic equations}.
Arch. Rational Mech. Anal. \textbf{103} (1988), no. 4, 319--345.

\bibitem{CMS} B. Choi, R. J. McCann and C. Seis,
               \emph{Asymptotics near extinction for nonlinear fast diffusion on a bounded domain}. 	arXiv:2202.02769.
 
\bibitem{DKenig} B. E. J. Dahlberg  and C. Kenig, 
       \emph{Nonnegative solutions the of the initial-Dirichlet problem for generalized porous medium equation in cylinders.} 
      J. Amer. Math. Soc. \textbf{1} (1988), 401--412.

\bibitem{DH} P. Daskalopoulos and R. Hamilton, 
       \emph{Regularity of the free boundary for the porous medium equation.} 
       J.  Amer. Math. Soc. \textbf{11} (1998), 899--965.
               
\bibitem{DHL} P. Daskalopoulos, R. Hamilton and K. Lee, 
       \emph{All time $C^\infty$-regularity of the interface in degenerate diffusion: a geometric approach.}
Duke Math. J. \textbf{108} (2001), no. 2, 295--327.
                      
\bibitem{DaK} P. Daskalopoulos and C.  Kenig, 
               Degenerate diffusions. Initial value problems and local regularity theory. EMS Tracts in Mathematics, 
               1. European Mathematical Society (EMS), Z\"urich, 2007. 


\bibitem{DiBenedetto} E. DiBenedetto, 
            \emph{Continuity of weak solutions to a general porous medium equation.} 
            Indiana Univ. Math. J. \textbf{32} (1983), no. 1, 83--118. 

           
\bibitem{DKV} E. DiBenedetto, Y. C. Kwong and V. Vespri,
          \emph{Local space-analyticity of solutions of certain singular parabolic equations}.
          Indiana Univ. Math. J. \textbf{40} (2) (1991), 741--765.



\bibitem{FGS} E. B. Fabes, N. Garofalo and S. Salsa, 
        \emph{A backward Harnack inequality and Fatou theorem for nonnegative solutions of parabolic equations.} 
        Illinois J. of Math. \textbf{30} (1986), 536--565.
        
\bibitem{FabS} E. B. Fabes and M. V. Safonov, 
        \emph{Behavior near the boundary of positive solutions of second order parabolic equations.} 
        J. Fourier Anal. and Appl. \textbf{96} (1997), 871--882.

\bibitem{FS} E. Feireisl and F. Simondon,  
         \emph{Convergence for semilinear degenerate parabolic equations in several space dimension.} 
         J. Dynam. Differential Equations \textbf{12} (2000), 647--673. 

\bibitem{GG} M. Giaquinta and E. Giusti,
         \emph{On the regularity of the minima of variational integrals}.
         Acta Math. \textbf{148} (1982), 31--46.

\bibitem{HLL} Q. Han, X. Li and Y. Li, 
         \emph{Asymptotic expansions of solutions of the Yamabe equation and the $\sigma_k$-Yamabe equation near isolated singular points.}
Comm. Pure Appl. Math. \textbf{74} (2021), no. 9, 1915--1970.

\bibitem{GW} M. Gr\"uter and K.-O. Widman, 
         \emph{The Green function for uniformly elliptic equations}.
Manuscripta Math. \textbf{37} (1982), no. 3, 303--342.

\bibitem{HK600} K. H\"ollig and H. O. Kreiss, 
          \emph{$C^\infty$-regularity for the porous medium equation}. 
          Math. Z. \textbf{192} (1986), no. 2, 217--224. 

\bibitem{JX19} T. Jin and J. Xiong,
          \emph{Optimal boundary regularity for fast diffusion equations in bounded domains}. 
          arXiv:1910.05160, to appear in Amer. J. Math.

\bibitem{JX20} T. Jin and J. Xiong,
          \emph{Bubbling and extinction for some fast diffusion equations in bounded domains}. 
          arXiv:2008.01311.
          
\bibitem{JX22} T. Jin and J. Xiong,
          \emph{Regularity of solutions to the Dirichlet problem for fast diffusion equations}. 
          arXiv:2201.10091.

\bibitem{KKV} C. Kienzler, H. Koch and J.L. V\'azquez, 
          \emph{Flatness implies smoothness for solutions of the porous medium equation.} 
          Calc. Var. Partial Differential Equations \textbf{57} (2018), no. 1, Art. 18, 42 pp.

\bibitem{Keldys} M. V. Keldy\u{s}, 
          \emph{On certain cases of degeneration of equations of elliptic type on the boundary of a domain.} 
          Doklady Akad. Nauk SSSR (N.S.) \textbf{77} (1951), 181--183.
          


\bibitem{KimLee} S. Kim and K.-A. Lee, 
          \emph{Smooth solution for the porous medium equation in a bounded domain.}
          J. Differential Equations \textbf{247} (2009), no. 4, 1064--1095.
          
\bibitem{Koch} H. Koch, 
           \emph{Non-Euclidean singular integrals and the porous medium equation.} 
           Habilitation. Universit\"at Heidelberg, Heidelberg, 1999.  

\bibitem{Knerr77} B. F. Knerr, 
           \emph{The porous medium equation in one dimension}. 
           Trans. Amer. Math. Soc. \textbf{234} (1977), no. 2, 381--415.

\bibitem{Kr} N. V. Krylov, 
         \emph{Boundedly inhomogeneous elliptic and parabolic equations in a domain.} 
         Izv. Akad. Nauk SSSR Ser. Mat. \textbf{47} (1983), no. 1, 75--108; 
         English translation: Math. USSR Izv. \textbf{22} (1984), no. 1, 67--98.
   
\bibitem{KMN} T. Kuusi, G. Mingione and K. Nystr\"om, 
         \emph{A boundary Harnack inequality for singular equations of parabolic type.} 
         Proc. Amer. Math. Soc. \textbf{142} (2014), 2705--2719. 
  
\bibitem{KwongY} Y. C. Kwong, 
          \emph{Interior and boundary regularity of solutions to a plasma type equation.} 
          Proc. Amer. Math. Soc. \textbf{104} (1988), no. 2, 472--478. 

\bibitem{KwongY3} Y. C. Kwong, 
          \emph{Asymptotic behavior of a plasma type equation with finite extinction}. 
          Arch. Rational Mech. Anal. \textbf{104} (1988), no. 3, 277--294.


\bibitem{KwongY2} Y. C. Kwong, 
          \emph{Boundary behavior of the fast diffusion equation}.
Trans. Amer. Math. Soc. \textbf{322} (1990), no. 1, 263--283.

\bibitem{LSU} O. A. Ladyzenskaja, V. A. Solonnikov and N. N. Ural'ceva, ``Linear and quasilinear equations of
parabolic type". Translated from the Russian by S. Smith. Translations of Mathematical Monographs,
Vol. 23 American Mathematical Society, Providence, R.I. 1968.


\bibitem{OKC}  O. A. Oleinik, A. S. Kalashnikov and Y.-I. Chzou,  
           \emph{The Cauchy problem and boundary problems for equations of the type of unsteady filtration.} 
           Izv. Akad. Nauk SSR Ser. Math. \textbf{22} (1958), 667--704.

\bibitem{OR} O. A.  Oleinik and E. V. Radkevic, 
         \emph{Second order differential equations with nonnegative characteristic form.} 
         Amer. Math. Soc., RI/Plenum Press, New York, 1973.
           
\bibitem{Sacks} P. E. Sacks, 
          \emph{Continuity of solutions of a singular parabolic equation.} 
          Nonlinear Anal. \textbf{7} (1983), no. 4, 387--409. 

\bibitem{Vaz2004} J. L. V\'azquez, 
      \emph{The Dirichlet problem for the porous medium equation in bounded domains. Asymptotic behavior}.
Monatsh. Math. \textbf{142} (2004), no. 1--2, 81--111.
                
\bibitem{Vaz} J. L. V\'azquez, 
        The Porous Medium Equation. Mathematical Theory. Oxford Mathematical Monographs, The Clarendon Press, Oxford University Press, Oxford, 2007. 

\bibitem{WWYZ} C. Wang, L. Wang, J. Yin, S. Zhou, 
      \emph{H\"older continuity of weak solutions of a class of linear equations with boundary degeneracy.} 
      J. Differential Equations \textbf{239} (2007), no. 1, 99--131.


\end{thebibliography}
\end{document}